\documentclass[final]{siamltex}

\usepackage{amsmath} 
\usepackage{amssymb}
\usepackage{dsfont}
\usepackage{graphicx}
\newtheorem{remark}{Remark}

\newcommand{\D}{{\, \rm d}}
\newcommand{\DD}{{\, \rm D}}

\usepackage{color}

\newcommand{\Div}{\mathrm{div}\;}
\newcommand{\R}{\mathbb{R}}

\newcommand{\bJ}{\mathbf{J}}
\newcommand{\bK}{\mathbf{K}}
\newcommand{\bI}{\mathbf{I}}
\newcommand{\bH}{\mathbf{H}}
\newcommand{\bL}{\mathbf{L}}
\newcommand{\bD}{\mathbf{D}}

\newcommand{\bn}{\mathbf{n}}

\newcommand{\dx}{\, {\mathrm d}}

\title{Analysis of a non-local and non-linear Fokker-Planck model for cell crawling migration}

\author{Christ\`ele Etchegaray\thanks{LMO, Univ. Paris-Sud, CNRS, Universit\'{e}
Paris-Saclay, 91405 Orsay, France.  \& 
MAP5, CNRS UMR 8145, Universit\'{e} Paris Descartes, 45 rue des Saints  P\`{e}res
75006 Paris,
France. ({\tt christele.etchegaray@math.u-psud.fr})}  \and Nicolas Meunier\thanks{MAP5, CNRS UMR 8145, Universit\'{e} Paris Descartes, 45 rue des Saints  P\`{e}res
75006 Paris,
France. ({\tt nicolas.meunier@parisdescartes.fr})} \and Raphael Voituriez.
        \thanks{Laboratoire de la mati\`ere condens\'ee, CNRS UMR 7600 \& Laboratoire Jean Perrin, UMR 8237, Universit\'e Pierre et Marie Curie, 4 Place Jussieu, 75255 Paris Cedex 05 France ({\tt voiturie@lptmc.jussieu.fr})}}

\begin{document}

\maketitle

\begin{abstract}
Cell movement has essential functions in development, immunity and cancer. Various cell migration patterns have been reported and a general rule has recently emerged, the so-called UCSP (Universal Coupling between cell Speed and cell Persistence), \cite{Maiuri}. This rule says that cell persistence, which quantifies the straightness of trajectories, is robustly coupled to migration speed. In \cite{Maiuri}, the advection of polarity cues by a dynamic actin cytoskeleton undergoing flows at the cellular scale was proposed as a first explanation of this universal coupling. Here, following ideas proposed in \cite{Maiuri}, we present and study a simple model to describe motility initiation in crawling cells. It consists of a non-linear and non-local Fokker-Planck equation, with a coupling involving the trace value on the boundary. In the one-dimensional case we characterize the following behaviours: solutions are global if the mass is below the critical mass, and they can blow-up in finite time above the critical mass. In addition, we prove a quantitative convergence result using relative entropy techniques. 
\end{abstract}



\section{Introduction}

Cell migration is a fundamental biological process involved in morphogenesis, tumor spreading, and wound healing \cite{Ridley_SignalFrontBack, Keller_Development, Huber_EMT}. One of its most spectacular instance is cell crawling, which is crucial to immune cells in order to reach an inflammation spot, but is also observed e.g. in tumor cells during metastasis formation \cite{metastasis}. Identifying key mechanisms involved in cell migration is then a major issue both for our fundamental understanding and for clinical research.

We focus here on cells crawling on a flat substrate, without any external cue. This type of movement occurs in several steps: protrusion of the cell at the "leading edge" (led by actin polymerisation), adhesion to the substrate, contraction and translocation of the cell body. However, if the notion of "leading edge" is clear for cells undergoing persistent motion (as in the chemotaxis phenomenon), the relation between protrusive activity and directionnality in more general cases is still unresolved. Indeed, migrating cells might transiently polarize - a front and a back appear. If such ability is weak, the resulting motion is close to random. This polarity is reflected at the molecular level by a restriction of certain molecules to particular regions of the inner cell membrane. For example, the activated forms of Rac and Cdc42 molecules are found at the front of the cell, whereas RhoA molecules are found toward the rear (see e.g \cite{Machacek_Rho_2009}). It is known that this asymmetry results from strong feedbacks exerted between these signaling pathways and mechanical elements \cite{polarity_mechanochemical}. Self-polarisation is then a multiscale phenomenon for which modeling efforts can bring new insights \cite{motility_number}.

Integrated models has been succesfully developped for addressing the issue of motile cell shape, in particular for the keratocyte \cite{multiscale_2D_keratocyte}. Computational models allow for in silico experiments and direct comparison with experimental measures (see \cite{Mogilner_Nature_2008,redundant_mogilner}, and more generally \cite{computational_shape}). In particular, it has been highlighted in \cite{redundant_mogilner} that redundant minimal mechanisms exist to ensure a motile stable shape for the keratocyte.

Self-polarisation has led to computational models investigating the role of contraction, adhesion and actin polymerisation for keratocytes or epithelial cells, providing useful predictions and allowing confrontation with experimental data  \cite{recho,balance_mogilner,polarization_mogilner}. Other mechanochemical models were built to perform in silico experiments (\cite{integrated1D,integrated2D}), but they remain unable to provide a minimal framework to explain the persistence issue in cell motion. 

Recently, in \cite{Maiuri}, it was shown experimentally that cell persistence is correlated to the flow of actin filaments from the front to the back of the motile part of the cell, as protrusions with faster actin flows are more stable in time. It was also shown that faster actin flows generate steeper gradients of actin-binding proteins. 

In this work, we study a conceptually minimal model for polarity initiation and maintenance based on an active gel description of the actin cytoskeleton, and its interaction with molecular signaling pathways. 
More precisely, we use a rigid version of the model first proposed in \cite{Blanch} that we enrich, in the spirit of \cite{Maiuri},  with a feedback loop between actin polymerisation locations and polarity markers \cite{Machacek_Rho_2009,Hall1998}, that are advected by a dynamic actin cytoskeleton undergoing flows at the cellular scale. 
This minimal model allows us to obtain qualitative results on cell persistence. 

We analyse the long time asymptotics of the model. The principal goal of our analysis is to identify regimes in which non homogeneous stationary states, that will be interpreted as polarized states, emerge. 

The main ingredients of the model are as follows. We assume that there exists a mesoscopic length scale, small compared to the cell but large compared to individual molecules at which the properties of the cytoskeleton and of the solvent, which constitutes the cell, can be described by continuous fields \cite{Joanny_Prost_2009}. Following \cite{Jul_Kru_Prost_Joanny, Jul_Kru_Prost_Joanny_2, Blanch}, we describe the actin cytoskeleton as a 2D Darcy flow, bounded by a membrane with a given shape. We assume that the relevant dynamics is sufficiently slow to neglect elastic effects. Actin is assumed to polymerize at the membrane and to depolymerize uniformly at a constant rate in the interior of the cell. Moreover, polymerisation is assumed to depend on the concentration of a marker that is advected by the actin flow itself. Finally, since we focus here on cell crawling on a flat substrate, the main external force arising is friction with the substrate. 

The markers, whose density is denoted $c(t, x)$, are assumed to diffuse in the cytoplasm and to be actively transported along the cytoskeleton. Consider a viscous active fluid with pressure $p$ filling a two-dimensional  bounded domain of fixed shape figuring the cell to describe the cytoskeleton. When all coefficients are set to 1, and after a global change of variable so that the domain of definition $\Omega$ is fixed, the resulting motion is a biased diffusion equation with advection field in the cell frame $u(t, x)=-\nabla p(t,x)$:
\begin{equation}\label{eq:modele_CD}
\partial _t c (t,x) = \Div \Big( \nabla p (t,x)  c(t,x) +\nabla c(t,x) \Big) \, ,\   t>0\, , \ x \in \Omega\, ,
\end{equation} 
together with zero-flux boundary condition on $\partial\Omega$ in order to conserve the molecular content. The active character of the cytoskeleton is contained in the pressure problem. Actin polymerization at the boundary leads to a Dirichlet condition depending on $c$, while depolymerization arises as a sink term inside the domain. More precisely, the pressure is assumed to satisfy, for all $t>0$:
\begin{equation}\label{eq:pression}
\begin{cases}
-\Delta p= -1\, , &\textrm{ on }\Omega\, ,\\
p=1- c \, ,  &\textrm{ on } \partial \Omega\, .\\
\end{cases}
\end{equation} 
Furthermore, considering a global friction coefficient set to $1$, the cell velocity $v$, which arises from the inside flow rubbing on the substrate, is given by
\begin{equation}\label{eq:cell_vel}
v(t)=- \int_{\partial \Omega} c(t,y) \bn \D \sigma\, , \quad t>0\, , 
\end{equation}
where $\bn$ denotes the outward unit normal and $\D \sigma$ the surface measure on $\partial \Omega$. 


In the one-dimensional case, assuming that $\Omega=(-1,1)$ the previously described model writes as a non-linear and non-local Fokker-Planck equation:
\begin{equation}\label{eq:modele_1D}
\partial _t c(t,x) =  \partial _{xx} c(t,x)+\partial_x\Big( \big(x + c(t,-1)-c(t,1)\big) c(t,x)\Big) \, ,\quad   t>0\, , \ x\in \Omega\, ,
\end{equation}
together with zero-flux boundary conditions at $x=-1$ and $x=1$:
\begin{equation}\label{eq:CL_modele_1D}
\begin{cases}
\big(c(t,-1)-c(t,1) -1\big) c(t,-1) + \partial _x c(t,-1) =0 \, , \\
\big(c(t,-1)-c(t,1)+1\big) c(t,1) + \partial _x c(t,1) =0  \, .
\end{cases}
\end{equation}
The main aim of this paper is to provide some results on the long time asymptotics of the solution to \eqref{eq:modele_1D} - \eqref{eq:CL_modele_1D} and to give estimates on the convergence rates in cases of convergence. To this end, we first look for stationary states. Let us denote $M = \int_{\Omega}c(t,x) \dx x$ the (conserved) mass of molecules inside the cell. In the case $M \le 1$, there exists a unique stationary state $G_M (x):=M exp(-x^2/2)/\int_{-1}^1exp(-y^2/2 )\D y $ towards which the solution converges and the asymptotic result is obtained through the convergence to zero of a suitable Lyapunov functional $\bL$ defined in \eqref{eq:lyapunov} in the case $M<1$, and through the convergence to zero of the relative entropy $\bH$ defined in \eqref{eq:relative_entropy} in the case $M=1$. Moreover, in both cases, using the logarithmic Sobolev inequality with a suitable function, we obtain an exponential decay to equilibrium.
\begin{theorem} \label{th:1D} 
Assume that the initial datum $c_0$ satisfies both $c_0 \in L^1(-1,1)$ and $\int_{-1}^1 c_0(x) \log c_0(x)\D x<+\infty$. Assume in addition that $M\le 1$, then there exists a global weak solution of \eqref{eq:modele_1D} - \eqref{eq:CL_modele_1D} that satisfies the following estimates for all $T>0$,
\begin{eqnarray*}
\sup_{t\in (0,T)} \int_{-1}^1 c(t,x) \log c(t,x)\D x &<& +\infty\, ,\\
\int_0^T\int_{-1}^1 c(t,x) \left( \partial_x \log c(t,x) \right)^2 \D x \D t &<& +\infty\, . 
\end{eqnarray*}
Moreover, the solution strongly converges in $L^1$ towards the unique stationary state $G_M (x) $ and the rate is exponential.
\end{theorem}


Solutions of \eqref{eq:modele_1D} may become unbounded in finite time (so-called blow-up). This occurs if the mass M is above the critical mass: $M > 1$, and for an asymmetric enough initial profile, measured with the first momentum shifted in $x=-1$ of $c$, $\tilde{\bJ}(t) = \int_{-1}^1   (x+1) c(t,x)\D  x$.
\begin{theorem}\label{th:BU}
Assume $M>1$ and the initial first moment shifted of 1 is small: $\tilde{\bJ}(0) <  (M-1)/2$. Assume in addition that $c_0$ satisfies $c_0(-1)-c_0(1)>1$. Then the solution to \eqref{eq:modele_1D} - \eqref{eq:CL_modele_1D}  with initial data $c(0,x) = c_0(x)$ blows-up in finite time. 
 \end{theorem} 
 
In Figure \ref{simus:cas_direct}, we illustrate Theorems \ref{th:1D} and \ref{th:BU}. 

Although in the present biological context, blow-up of solutions is interpreted as cell polarisation, such a blow-up in finite time might be questionable. Indeed a strong instability drives the system towards an inhomogeneous state and blow-up corresponds to the case where the drift becomes infinite. The boundary condition \eqref{eq:CL_modele_1D}  which is responsible for infinite drift turns out to be unrealistic from a biophysical viewpoint. On the way towards a more realistic model, we distinguish between cytoplasmic content $c(t, x)$ and the concentration of trapped molecules on the boundary at $x = \pm 1$ : $\mu_\pm (t)$. Then the  exchange of molecules at the boundary is described by very simple kinetics.

\begin{figure}\label{simus:cas_direct}
\centering
\includegraphics[scale=0.3]{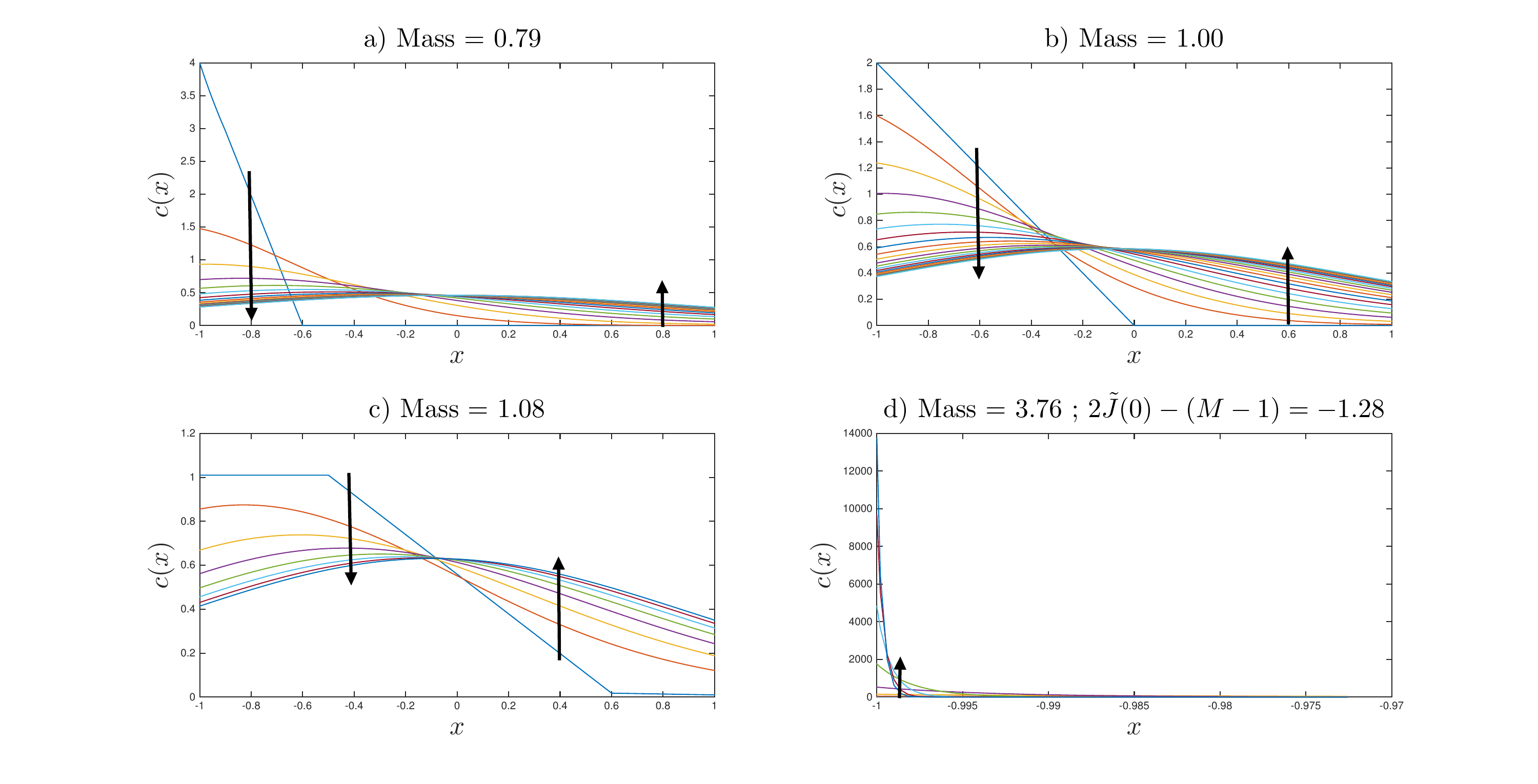}
\caption{Numerical simulations of the spatial concentration profile $c(x)$ of the marker. Each plot corresponds to a different initial profile and mass :  a) sub-critical case ; b) critical case ; c), d) super-critical case. Each curve represents the concentration profile at a specific time. Parameters: $T=4$ ; $dt=10^{-2}$ ; $dx = 2* 10^{-3}$.}
\end{figure}


The plan of this work is as follows. In Section \ref{sec:model}, we introduce the model. In Section \ref{sec:1D}, we analyse with full details
the one-dimensional case. In Section \ref{sec:ODE/PDE}, we briefly study a model with exchange of markers at the boundary in the one-dimensional case.

\section{The model}\label{sec:model}

Our purpose in this section is to derive the model given by equations \eqref{eq:modele_CD} - \eqref{eq:pression} - \eqref{eq:cell_vel} which describes the behaviour of an active viscous fluid featuring the cytoskeleton. 

\subsection{Main constitutive assumptions}

We consider a two-dimensional layer of viscous fluid, representing the cell cytoskeleton, surrounded by a rigid membrane. Following \cite{Jul_Kru_Prost_Joanny, Jul_Kru_Prost_Joanny_2, Blanch}, to model actin polymerisation and depolymerization,  we add active properties to the viscous fluid. The first active property we consider is the out-of-equilibrium polymerization of the fluid. In a cell, actin monomers are added to actin filaments by the consumption of the biological fuel ATP. Other proteins regulate the nucleation and polymerization of actin filaments. It is commonly observed that actin polymerization activators such as WASP proteins preferentially locate along the cell membrane \cite{Ridley_LeadingEdge}. For this reason we suppose that the fluid is polymerised at the membrane. 
The main novelty of our work relies on the coupling between actin polymerization and a biological marker which is transported by actin flows. Its aggregation in a part of the membrane characterizes the rear of the cell, hence its polarisation. This marker could be an antagonist to polymerization-inducing molecules (Rac1, Cdc42), such as RhoA, Arpin, or even myosin II (see \cite{arpin, Gautreau_review}).
Furthermore polymerization is balanced by depolymerization, which we assume to occur uniformly at a constant rate in the cell body, to ensure the renewal of ressources for polymerisation. polymerization and depolymerization induce an inward flow which rubs on the substrate. This friction is responsible for the cell displacement. 

 \begin{figure}\label{fig:UCSP}
 \centering
 \includegraphics[scale=0.25]{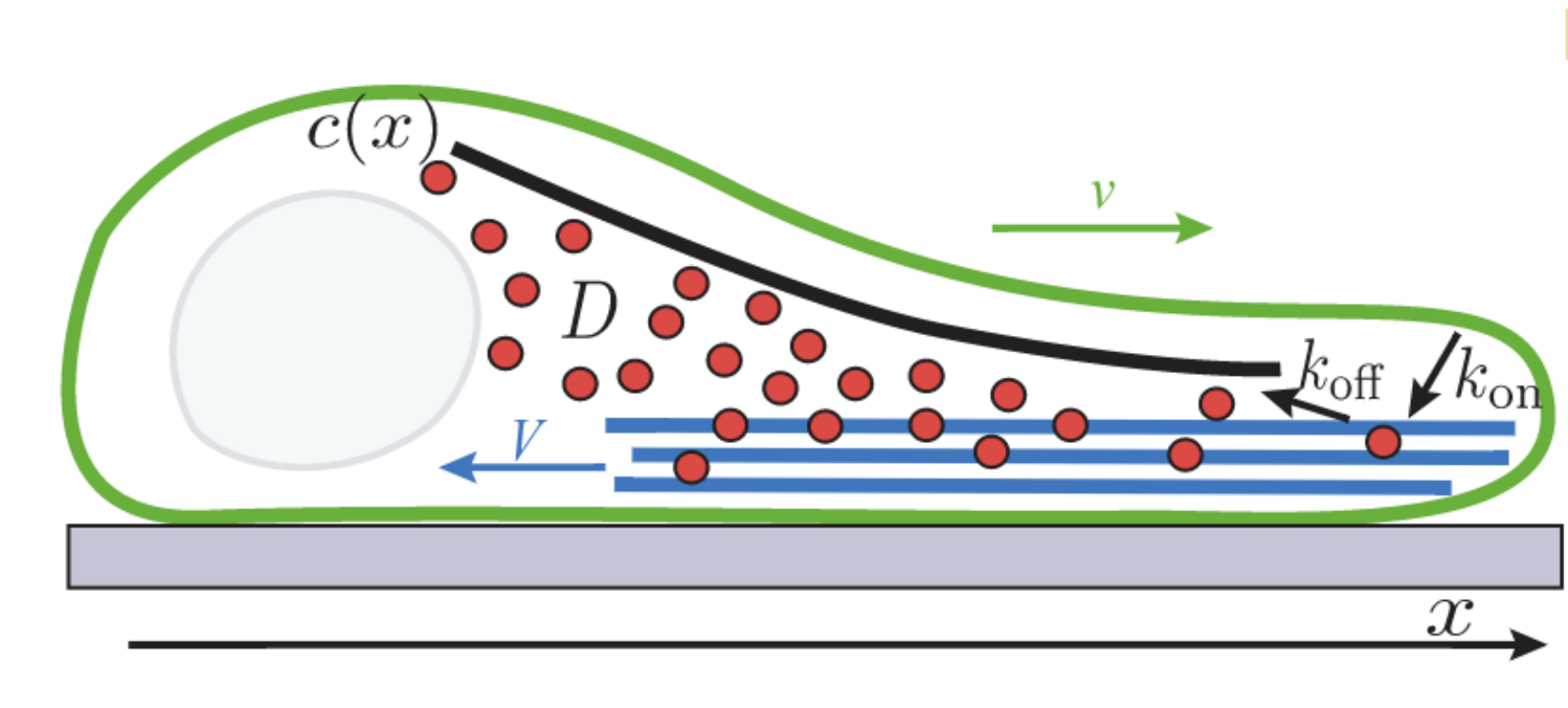}
 \caption{Interaction between actin flows and a molecular polarity marker \cite{Maiuri}.}
 \end{figure}

More precisely, we consider an active Darcy flow, which models the actin cytoskeleton, inside a (moving) domain $\Omega(t)\subset \R^2$, where $\Omega(0)$ is a disc of radius $R>0$. The domain moves by translation with a velocity $V(t) \in \R^2$. We introduce the fluid pressure $\tilde{P}(t,X)$, and $\tilde{U}(t,X)$ the actin filaments velocity. Finally, we denote $\tilde{C}(t,X)$ the concentration of molecular marker. All these quantities are defined for $X\in \Omega(t)$ and $t>0$. We assume that the fluid density $\rho(t,X)\equiv \rho$ is constant in time and space. Then, the fluid problem writes
\begin{equation}\label{eq:poisson}
\begin{cases}
\Div (\tilde{U})=-\frac{1}{\xi}\Delta \tilde{P}  = -k_d & \mbox{ in } \Omega(t)\, ,\\
\tilde{P}  = k _p  & \mbox{ on } \partial \Omega(t)\, .
\end{cases}
\end{equation}
The boundary condition accounts for polymerization at the edge of the cell, while $-k_d$ describes depolymerization in the cell body. 

\begin{remark}
Note that equation \eqref{eq:poisson} is similar but different from the model first introduced in \cite{Blanch}. Indeed in our case the polymerization is modeled through a pressure term and not a velocity one. 
\end{remark}

In the limit of low Reynolds number, viscous forces dominate over inertial forces and the Navier-Stokes equation simplifies to the force balance principle:
\begin{equation}\label{Newton}
- \Div \sigma  = f \qquad \textrm{ in } \Omega(t) \, ,
\end{equation} 
where  $\sigma $, the stress tensor, is given by 
\begin{equation}\label{newtonian}
\sigma = \mu \left( \nabla \tilde{U} + {}^t\nabla \tilde{U}\right) -\tilde{P} \textrm{ Id}\, ,
\end{equation}
with $\mu$ being the viscosity. 
Since the layer is placed on a substrate, we consider a friction force
\begin{equation}\label{body_force}
f  = -\xi \tilde{U}\, ,
\end{equation} 
where $\xi$ is an effective friction coefficient.

 Following \cite{Callan-Jones_Joanny_Prost}, we neglect viscosity arising from the polymer-polymer and polymer-solvent friction forces and consider the limit $\mu\to 0$. This reduces to 
\begin{equation}\label{eq:darcy}
\nabla \tilde{P}=-\xi \tilde{U}\qquad \textrm{ in } \Omega(t)\, .
\end{equation}
Notice that the pressure, hence the friction force, acts for the integrity of the fluid (no phase separation).

\begin{remark}
If we were considering a deformable cell, see \cite{Blanch}, it would have been natural to assume that the pressure $\tilde{P}$ satisfies
$\tilde{P}= -\gamma \kappa$  in $\partial \Omega(t)$,
where $\kappa$ is the curvature of $ \partial \Omega(t)$ and $\gamma \geq 0$ is the superficial tension of the cell membrane. Here we consider a caricatural situation since we assume that the cell domain is non-deformable and we do not impose any additional condition on $\tilde{P}$ on the boundary of $\Omega(t)$. 
\end{remark}

\begin{remark}{About the polymer density:}
this model can be derived from a more
realistic one, where polymerization and depolymerization locally modify the polymer
density. More precisely, write
\begin{equation}
\partial_t \rho + \Div(\rho \tilde{U}) = - k_d \rho\, ,
\end{equation}
where $\Omega(t) = \{\rho(t,\cdot) >0 \}$ denotes the region occupied by the cytoskeleton. To account for polymerization that takes place at the edge of the cell and which consists in a local increase in actin concentration, we impose on the cell membrane a jump on the actin concentration:

\begin{equation}
\rho = \rho_0  + \varepsilon k_p \quad \text{ on } \partial \Omega(t)\,,
\end{equation}
where $\varepsilon$ is a small parameter and $\rho_0$ a constant. One can also assume that the polymerization rate is a continuous non-zero polar function in an annulus initially defined as $\overline{\Omega}(0)\setminus B(0,R-\lambda)$, of width $\lambda >0$, see e.g \cite{Kru_Joanny_Jul_Prost_contractility}.\par 
As it is classical, see \cite{Callan-Jones_Voituriez}, in addition, we assume that $\tilde{P}$ is a function of $\rho$:
\begin{equation*}
\tilde{P}=\frac{1}{\varepsilon}(\rho-\rho_0)\, ,
\end{equation*}
we then get the following problem:
\begin{eqnarray*}
\partial_t \rho -\frac{1}{\xi}\Div\left(\rho \nabla \left(\frac{\rho}{\varepsilon}\right ) \right )  &= -k_d \rho & \qquad \mbox{ in } \Omega\, ,\\
\rho  &= \rho_0 + \varepsilon k _p  & \qquad \mbox{ on } \partial \Omega\, .
\end{eqnarray*}
Formally the limit $\varepsilon \to 0$ of the previous model leads to the Poisson problem \eqref{eq:poisson}. This amounts to saying that the osmotic pressure $\tilde{P}$ ensures at all time that the polymer density stays constant. The rigorous justification of this limit is not our purpose here, and will be the object of a future work. 
\end{remark}

Finally, we have to prescribe the domain velocity arising from friction forces. Since the gel layer is at mechanical equilibrium, the cell moves as a consequence of the inside flow rubbing on the substrate, hence we will consider a moving domain. Friction forces occur at the microscopic scale, and mesoscopic tension forces also are at play. However, for the sake of simplicity, we will neglect this heterogeneity to consider a global friction coefficient $\gamma$. We write for all $t>0$ 
\begin{equation}\label{velocity}
V(t) = -\gamma \int_{\Omega(t)} \tilde{U}(t,X) \dx X.
\end{equation}

The main novelty of our work is to consider that $k_p$ is a function of $c$, the concentration of a biological marker. This marker is assumed to diffuse and to be transported by the actin filaments, modeled by the previously described fluid. 
At the boundary, we prescribe a zero flux condition to ensure mass conservation. The corresponding problem writes
\begin{equation*}
\partial_t \tilde{C}(t,X) + \Div  \left( \tilde{U}(t,X) \tilde{C}(t,X) - D \nabla \tilde{C}(t,X)\right) = 0 \quad\text{ for } X\in \Omega(t), \, t>0\,  ,
\end{equation*}
with the zero-flux boundary condition
\begin{equation*}  
\left(D \nabla \tilde{C}(t,X) - \tilde{C}(t,X) \tilde{U}(t,X)  \right)\cdot \bn =  0 \quad \text{ for } X\in  \partial \Omega(t), \, t>0  \, ,
\end{equation*}
where $\bn$ is the outward unit normal to the boundary.
The zero flux boundary condition ensures that the total mass is preserved: $\frac{\dx}{\dx t} \int_{\Omega(t)} \tilde{C}(t,X) \dx X=0$.

\begin{remark}
The non-deformability of $\Omega(t)$ is an important downside of the model, that conceals fundamental modelization issues. Indeed, the cell velocity is defined globally, preventing the description of local effects: experimentally, it is observed that the activity of the cytoskeleton deforms the membrane, leading to a displacement while exerting a geometric feedback on actin flows.
\par 
Notice that our choice for the friction force amounts to consider friction arising from both the retrograde flow and the displacement, but their effects on the equations are the same in 1D. 
Hence, the role of motion in the initiation and maintenance of polarisation cannot be investigated properly, but this will be studied in future works in 2D and for a free-boundary model.
\end{remark}

In the spirit of \cite{Maiuri}, we will assume that polymerization occurs more likely at the boundary points where the marker concentration is the lower (the marker we consider is a rear marker). 

Recalling what was explained previously, the fluid problem writes
\begin{equation*}
\begin{cases}
\nabla \tilde{P}(t,X) =- \xi \tilde{U}(t,X) &  \text{ for } X\in \Omega(t), \, t>0 \, ,\\
\Div  \tilde{U}(t,X) = - K_d&  \text{ for } X\in \Omega(t), \, t>0 \, , \\
\tilde{P}(t,X) = f\left(\tilde{C}(t,X)\right) :=  \alpha -\beta \tilde{C}(t,X)&  \text{ for } X\in \partial \Omega(t), \, t>0 \, , 
\end{cases}
\end{equation*}

\begin{remark}
A more realistic boundary condition on $\tilde{P}$ would have been to impose that
\begin{equation*}
\tilde{P}(t,X) =  \Big(\alpha -\beta \tilde{C}(t,X) \Big)_+ \text{ for } X\in \partial \Omega(t), \, t>0 \, ,
\end{equation*}
where $(\cdot )_+$ denotes the positive part. Indeed the case where $\tilde{P}$ takes negative values is not realistic from the modelling viewpoint. 
\end{remark}

\begin{remark}\label{rem:autre_vit}
Another option for the cell velocity would be to choose
\begin{eqnarray*}
V_2(t) &=& -\gamma \int_{\Omega(t)} \nabla \tilde{C}(t,X) \dx X = -\gamma \int_{\partial \Omega(t)} \tilde{C}(t,X) \bn \dx \sigma \, .
\end{eqnarray*}
However this latter choice is less relevant from a biological viewpoint, since the global cell displacement is a consequence of the actin flow rubbing on the substrate. 
\end{remark}

\paragraph{Problem formulation on a fixed domain}
We now define the problem on a fixed domain $\Omega_0 := \Omega(0)$. In such a case the domain motion appears in the problem formulation. 
As $(\tilde{P},\tilde{U},\tilde{C})$ are functions on $\Omega(t)$, we denote $(P,U,C)$ their analogous on $\Omega_0$. Moreover we define the following map: 
\begin{displaymath}
\begin{array}{lrcl}
L(.,t) :  & \Omega_0 &\longrightarrow&\Omega(t) \\
 & x & \mapsto  & X = x + \int_0^t V(s) \dx s = L(t,x),
\end{array}
\end{displaymath}
which gives $C(t,x) = \tilde{C}(t,L(t,x))$ (for example for $C$).

Let us now rewrite the fluid problem on $\Omega_0$. To do so we observe that for all $x\in  \Omega_0$ and for all $t>0$:
\begin{equation*}
\partial_t C(t,x) = \partial _t\tilde{C}(t,L(t,x)) + \underbrace{\partial _t L(t,x)}_{= V(t)}  \nabla  _X \tilde{C}(t,X)=\frac{\DD \tilde{C}}{\DD t},  
\end{equation*}
where $\frac{\DD }{\DD t}$ denotes the total derivative. The convection-diffusion equation rewrites as
\begin{displaymath}
\left\lbrace
\begin{array}{rclr}
\displaystyle \frac{\DD}{\DD t}\tilde{C}(t,X) +  \Div \left( (\tilde{U}(t,X)-V(t)) \tilde{C}(t,X) - D \nabla \tilde{C}(t,X)\right) &=& 0 &\,  X\in \Omega(t), \, t>0 ,\\
\displaystyle \left( D \nabla \tilde{C}(t,X) - \tilde{C}(t,X)(\tilde{U}(t,X)-V(t))  \right) \cdot \bn &=&  0&\,  X\in  \partial \Omega(t), \, t>0 .
\end{array}\right.
\end{displaymath}
Moreover, the Jacobian matrix related to $L$ is the identity matrix (as $V$ is space-independent), leading to $\nabla_X \tilde{C}(t,X) = \nabla_x C(t,x)$. This, combined with $\tilde{U}(t,X) = U(t,x)$, yields that
\begin{displaymath}
\left\lbrace
\begin{array}{rcll}
\displaystyle \partial_t C(t,x) + \textrm{div } \left[ (U(t,x)-V(t)) C(t,x) - D \nabla C(t,x)\right] &=& 0 &\, \text{ for } x\in  \Omega_0, \, t>0 ,\\
\displaystyle \left( D \nabla C(t,x) - C(t,x) (U(t,x)-V(t))  \right) \cdot \bn &=&  0& \, \text{ for } x\in  \partial \Omega_0, \, t>0 .
\end{array}\right.
\end{displaymath}
We can check again mass conservation: $\frac{\dx}{\dx t} \int_{\Omega_0} C \dx x=0$.

\paragraph{Nondimensionalization}
Let us introduce the following typical quantities for nondimensionalization: 
$\overline{L} = 2R$, where $R$ is the radius of $\Omega_0$, $\overline{P} = \alpha= \frac{\overline{F}\, \overline{L}^2}{\overline{L}}$, $\overline{F}=\frac{\alpha}{\overline{L}}=\frac{\alpha}{2R}=\xi \overline{V}$, $\overline{V}=\frac{\alpha }{2R\xi }$, $\overline{T} = \frac{\overline{L}}{\overline{V}} = \frac{4R^2 \xi}{\alpha }$ and $\overline{C}=\overline{L}^{-2}$.

\begin{remark}
It has to be noticed that the force we consider is actually a force per unit of area. This explains the expression used for the typical pressure $\overline{P}$, corresponding to a force divided by a length (in 2D).
\end{remark}
\begin{remark}
The typical pressure corresponds to a maximal incoming pressure of the fluid at the boundary of the domain.
\end{remark}

Now, write $(p,u,c,v)$ the nondimensionalized analogous to $(P,U,C,V)$. 
The nondimensionalized problem writes: 

\begin{equation}\label{fluid_adim}
\begin{cases}
            \nabla p = - u          & \text{ on } \Omega_0\, ,   \\
            \Div u  = - \frac{\xi4R^2}{\alpha} K_d =: -k_d  & \text{ on } \Omega_0\, ,  \\
			p = 1-\delta c  &  \text{ on } \partial \Omega_0\, , 
\end{cases}
\end{equation}
where $\delta:=\frac{\beta \overline{C}}{\alpha}$ and the domain velocity is: 
\begin{equation}\label{vcell_adim}
 v(t) = -\gamma \int_{\Omega_0} u\,  \dx x\, .
\end{equation}

The convection-diffusion problem is
\begin{equation}\label{conc_adim}
\begin{cases}
 \partial_t c + \Div  \left(  (u-v) c - D' \nabla c\right) = 0 &\, \quad \text{ on } \Omega_0\, ,\\
 \left(D' \nabla c - (u-v) c\right) \cdot \bn = 0& \, \quad  \text{ on } \partial \Omega_0\, , 
\end{cases}
\end{equation}
where we have introduced $D' = \frac{D}{\overline{L}\, \overline{V}} =  \frac{D\xi }{\alpha}$.
Moreover, the global mass is prescribed: 
\begin{equation}\label{mass}
\int_{\Omega_0} c \,  \dx x = M\, .
\end{equation}

\subsection{The one-dimensional case}

The first equation of (\ref{fluid_adim}) rewrites 
\begin{equation}\label{eq:vitesse_1D}
u(t,x) = -\partial_x p(t,x)  \, ,
\end{equation}
and using \eqref{vcell_adim}, it leads to 
\begin{equation}\label{velocity_pressure}
v(t) = \gamma(p(t,b)-p(t,a))\, .
\end{equation}
Moreover, differentiating \eqref{eq:vitesse_1D} and using that $\partial _x u= -k_d$,  we find $\partial_{xx}p(t,x) = k_d$ and we can write
\begin{equation}
p(t,x) = \frac{k_d}{2}(x-a)^2 + d(x-a)+p(t,a)\, ,
\end{equation}
where
\begin{displaymath}
d =  \frac{p(t,b)-p(t,a)}{b-a}-\frac{k_d}{2}(b-a)\, .
\end{displaymath}
Now, replacing $p$ in the expression of $u$ this gives 
\begin{equation*}
u(t,x) = -k_d \left(x-\frac{a+b}{2}\right) - \frac{p(t,b)-p(t,a)}{b-a}\, .
\end{equation*}

Finally recalling the boundary conditions in (\ref{fluid_adim}), we obtain the following non-linear and non-local convection-diffusion equation
\begin{equation}\label{mod:1D_renorm}
\begin{cases}
\partial _t c &= \partial_x\left( (v-u) c + D' \partial _x c\right)  \qquad \textrm{ on } (a,b), \textrm{ for } t>0\, ,\\
0&=(v-u) c + D' \partial _x c \qquad \textrm{ on } \{a,b\}, \textrm{ for } t>0\, ,
\end{cases}
\end{equation}
with a velocity $v-u$ given by: for all $x\in (a,b)$ and  for all $ t>0$
\begin{equation}\label{u}
v(t) - u(t,x) = k_d \left(x-\frac{a+b}{2}\right) +\left(\delta \gamma + \frac{\delta }{b-a}\right)(c(t,a)-c(t,b))\, ,
\end{equation}
and the domain velocity:
\begin{equation*}
v(t) = \gamma \delta \left(c(t,a)-c(t,b)\right)\, ,
\end{equation*}
with $\delta = \frac{\beta \overline{C}}{\alpha}>0$ and we recall that $D' = \frac{D\xi }{\alpha }\geq 0$.

\begin{remark}
Note that up to a constant we recover the cell velocity $v_2$, see Remark \ref{rem:autre_vit}.
\end{remark}

From now on, in order to simplify the computations we assume that $\Omega=(-1,1)$. In such a case the velocity simply rewrites as 
\begin{equation}\label{eq:vit_dim1}
v(t) - u(t,x) = k_d x +\left(\delta \gamma +  \frac{\delta}{2}\right) (c(t,-1)-c(t,1)).
\end{equation}
Finally, taking $k_d=D'=1$ and $\delta \left(\gamma + \frac{1}{2}\right)=1$, we obtain the problem \eqref{eq:modele_1D}-\eqref{eq:CL_modele_1D}.

We find a model that presents some similarities with the model studied in \cite{HBPV,CalvezMeunier,CalvezMeunier_Siam,Muller_Plos, Lepoutre_Meunier_Muller_JMPA} to describe yeast cell polarisation. The main difference here is the presence of an additional drift towards the cell center.

\section{The boundary non-linear Fokker-Planck equation in dimension 1}\label{sec:1D}

In this part we study the non-local and non-linear Fokker-Planck equation \eqref{eq:modele_1D} - \eqref{eq:CL_modele_1D} that we recall now: 
\begin{equation}\label{eq:1D}
\begin{cases}
\partial _t c(t,x) =  \partial _{x} \Big(\partial_x c(t,x)+\left(x + c(t,-1)-c(t,1)\right) c(t,x)\Big) \, , \\
\partial _x c(t,-1) +\left(c(t,-1)-c(t,1) -1\right) c(t,-1) =0 \, , \\
\partial _x c(t,1)+\left(c(t,-1)-c(t,1)+1\right) c(t,1)   =0  \, ,
\end{cases}
\end{equation}
and we prove Theorems \ref{th:1D}  and \ref{th:BU}. 

We are looking for a solution to \eqref{eq:1D}. As it is classical we start by giving a proper definition of weak solutions, adapted to our context:
\begin{definition}\label{def:weak}
We say that $c(t,x)$ is a weak solution of (\ref{eq:1D})  on $(0,T)$ if it satisfies:
\begin{equation}
c\in L^\infty\left(0,T;L^1_+(-1,1)\right)\, , \quad \partial_x c \in L^1\left((0,T)\times (-1,1)\right)  \, , \label{eq:flux L1}
\end{equation}
and $c(t,x)$ is a solution of (\ref{eq:1D}) in the sense of distribution in $\mathcal D'(-1,1)$.
\end{definition}

Since the flux $\partial_x c(t,x) + \left( c(t,-1)-c(t,1) +x\right) c(t,x)$ belongs to $ L^1\left((0,T)\times (-1,1)\right)$, the solution is well-defined in the distributional sense under assumption (\ref{eq:flux L1}). In fact we can write
\[\int_0^T \left( c(t,-1)-c(t,1)\right) \D t   = - \int_0^T\int_{-1}^1 \partial_x c(t,x)\D x\D t\, .\]
Moreover, one has
\[2c(t,\pm 1) = \int _{-1}^1 (x\pm 1) \partial_x c(t,x)\D x  + \int _{-1}^1 c(t,x)\D x\, .\]

Note also that to prove existence of weak solutions in the sense of Definition \ref{def:weak}, one should perform a regularization procedure. In this work we focus on long-time dynamics and we will not detail such a regularization procedure which is classical, we refer to \cite{CalvezMeunier_Siam} for more details. 

Let us now observe that the non-negativity of a solution is preserved. Indeed if $c$ is solution in $L^1_x$, since $\mbox{sgn}(c) \partial_{xx}^2 c \le \partial_{xx}^2 |c|$, then 
\begin{align*}
  \frac{\mathrm d}{\mathrm dt} \int_{-1}^1 \left( |c| - c \right) \D x  \le 0. 
\end{align*}
This proves that if $|c_0| = c_0$ almost everywhere (initial data
non-negative) then $|c|=c$ almost everywhere for later times.

Moreover, weak solutions in the sense of Definition \ref{def:weak} are mass-preserving:  $M =\int_{-1}^1 c_0(x)\D  x = \int_{-1}^1 c(t,x)\D  x$. 
A simple computation on the first momentum defined by $\bJ(t) =\int_{-1}^1 x c(t,x)\D x$, leads to 
\begin{equation}\label{eq:first_momentum}
\frac{\D}{\D t} \bJ(t)=(1-M)\left(c(t,-1)-c(t,1)\right)-\bJ(t)\, .
\end{equation}
Note that $\bJ(t) \ge -M$ since $\int_{-1}^1 (x+1) c(t,x)\D x \ge 0$.
Here, similarly to \cite{CalvezMeunier, CalvezMeunier_Siam}, we will prove that the following dichotomy occurs:
\begin{itemize}
\item when $M\le 1$, the linear Fokker-Planck part drives the equation and the solution converges toward the unique stationary state, see sections~\ref{sec:M<1} and~\ref{sec:M=1},
\item when $M>1$, the equation admits three stationary states, two of them being non-homogeneous. Moreover, for an asymmetric enough initial condition, the solution blows up in finite time, see section~\ref{sec:M>1}.
\end{itemize}

\subsection{Global existence and asymptotic analysis for sub-critical mass $M<  1$ }\label{sec:M<1}

In this section we prove Theorem \ref{th:1D} in the case $M<1$. The proof is decomposed in three steps. First we compute the stationary state of \eqref{eq:1D}, then we build a Lyapunov functional, this allows establishing \emph{a-priori estimates}. Finally we investigate long-time behaviour of solutions in the case $M<1$ using entropy methods. We stress out that the method for proving global existence in this case strongly relies on the existence of a Lyapunov functional, Lemma \ref{lem:entropy u}. This is why we analyse the global existence and the long time behaviour simultaneously. Note that the proof is very close to the one given in \cite{CalvezMeunier_Siam} and in \cite{Lepoutre_Meunier_Muller_JMPA}. The differences concern the expression of the stationary state and the domain geometry. In \cite{CalvezMeunier_Siam, Lepoutre_Meunier_Muller_JMPA}, the stationary state is the family $\{\alpha \exp\left( -\alpha x -x^2/2\right)\}_{\alpha}$, and the domain is the half-line. 

Let us start with the existence of a unique stationary state. 
\begin{lemma}\label{lemma:ESsscritique}
If $M<1$ equation \eqref{eq:1D} admits a unique stationary state given by $G_M=M\exp(-x^2/2)/\int_{-1}^1 \exp(-y^2/2) \D y$.
\end{lemma}
\begin{proof}
An easy computation shows that any stationary state for \eqref{eq:1D} is either $G_M$ (which is symmetric) or of the form $G_{\alpha}$ for $\alpha = G_{\alpha}(-1)- G_{\alpha}(1)\neq 0$ to be found. For $\alpha >0$, it writes
\begin{equation} \label{eq:stat state rescaled}
G_{\alpha}(x) =  \frac{\alpha }{1-\exp(-2\alpha)} \exp\Big(-\alpha (x+1) - \frac{x^2-1}{2}\Big)\, .
\end{equation}  

It remains to find $\alpha$ such that the mass constraint $\int_{-1}^1 G_{\alpha}(x)\D x = M$ is satisfied. This rewrites $P(\alpha) = M$, $P$ being the function defined by:
\begin{eqnarray*}
P(\alpha) &=& \int_{0}^{2\alpha}  \frac{1 }{1-\exp(-2\alpha)} \exp\left(-y - \frac{1}{2}\left(\big(\frac{y}{\alpha}-1\big)^2-1\right)\right) \D  y\, . 
\end{eqnarray*}
We observe that $ P(\alpha) > \int_{0}^{2\alpha}  \frac{1 }{1-\exp(-2\alpha)} \exp(-y )\D  y = 1$, hence  there is no stationary state of the form $G_{\alpha}$ with $\alpha >0$. 

The case $\alpha <0$ is done similarly.
\end{proof}

\paragraph{Lyapunov functional}
As it is classical we note the relative entropy
\begin{equation}\label{eq:relative_entropy}
\bH (u|v) = \int_{-1}^1 u(x) \log \left(\frac{u(x)}{v(x)} \right) \D x,
\end{equation}
and the Fisher information
\begin{equation*}
\bI (u|v) = \int_{-1}^1  u(x)\left(\partial_x\left(\log \frac{u(x)}{v(x)}\right)\right)^2 \D x.
\end{equation*}
Note that $\bH(c|G_M) \ge M\log M$ for all $t>0$ by Jensen's inequality. 
We introduce a Lyapunov functional for equation \eqref{eq:1D}:
\begin{equation}\label{eq:lyapunov}
\bL(t)=\bH(c|G_M)  + \frac{\bJ(t) ^2}{2(1-M)}   \, ,
\end{equation}
where $\bJ$ is given by \ref{eq:first_momentum}.
Let $\Gamma_{c}$ be defined by
\begin{equation}\label{Lyapounov_ss_crit_1}
\Gamma_{c}(x) =  A_c\exp \left(-\left(c(t,-1)-c(t,1) \right) \, x -\frac{x^2}{2}\right)\, ,
\end{equation}
with 
\begin{equation}\label{cte_Lyapounov_ss_crit_1}
A_c=\frac{M}{\int_{-1}^1 \exp\left( -\left(c(t,-1)-c(t,1) \right) \, y -\frac{y^2}{2}\right)\D y}\, .
\end{equation}

\begin{lemma} \label{lem:entropy u}
For the problem \eqref{eq:1D}, if $M<1$, then the Lyapunov functional $\bL$ is non-increasing:
\begin{equation}\label{eq:dissipation u_0}
\frac{\D}{\D t} \bL(t) = - \bD(t) \leq 0\, , 
\end{equation}
where the dissipation is
\begin{equation}\label{eq:dissipation u}
\bD(t) = \bI(c|\Gamma_{c}) + \frac{1}{(1-M)}\Big((c(t,-1)-c(t,1))(1-M)-\bJ(t)\Big)^2\, .
\end{equation}
\end{lemma}

\begin{proof}
We compute the evolution of the entropy:
\begin{eqnarray*}
& &\frac{\D }{\D t} \bH(c|G_M)(t)= \int_{-1}^1 \partial _t c(t,x)\left( \log (c(t,x))   + \frac{x^2}{2}  + 1 - \log \left(\frac{M}{\int_{-1}^1 e^{-\frac{y^2}{2}} \D y}\right) \right) \D  x \, , \\
\end{eqnarray*}
where $\int_{-1}^1 \partial _t c(t,x)( 1 - \log(M/\int_{-1}^1 e^{-\frac{y^2}{2}} \D y )) \D  x =0$ by mass conservation. Hence,
\begin{eqnarray}
&&\frac{\D }{\D t} \bH(c|G_M)(t) =  - \int_{-1}^1  \bigg( \partial_x c(t,x) + \big(x+c(t,-1)-c(t,1)\big) c(t,x)  \bigg) \left( \frac{ \partial_x c(t,x)}{c(t,x)}   + x \right) \D  x \nonumber \\
& =&  - \int_{-1}^1 c(t,x) \Big( \partial_x \log c(t,x) + x  \Big)^2\D  x +
\left(c(t,-1)-c(t,1)\right)^2 - \left(c(t,-1)-c(t,1)\right) \bJ(t) \nonumber \\
& = & - \int_{-1}^1 c(t,x) \bigg( \partial_x \log c(t,x) + x +c(t,-1)-c(t,1) \bigg)^2\D  x \nonumber\\
& &\qquad +
(M-1)\left(c(t,-1)-c(t,1)\right)^2 + \left(c(t,-1)-c(t,1)\right) \bJ(t) \, . \label{eq:u0}
\end{eqnarray}
We can eliminate $c(t,-1)-c(t,1)$ from (\ref{eq:u0}) in the two following steps:
\begin{eqnarray}
\left(c(t,-1)-c(t,1)\right) \bJ(t) &=& 
\frac{\bJ(t)}{(1-M)}\frac{\D }{\D t} \bJ(t)+ \frac{\bJ(t)^2 }{(1-M)} \nonumber \\
& =& -\frac{\D }{\D t}  \frac{ \bJ(t)^2}{2(1-M)}  
+ \frac{2\bJ(t)}{(1-M)}\frac{\D }{\D t} \bJ(t)+ \frac{\bJ(t)^2}{(1-M)}  \, , 
\label{eq:entropy J}
\end{eqnarray} 
leading to
\begin{equation} 
- \frac{1}{(1-M)}\left(\frac{\D }{\D t} \bJ(t)\right)^2   
=(M-1)\left(c(t,-1)-c(t,1)\right)^2 +\frac{2\bJ(t)}{(1-M)}\frac{\D }{\D t} \bJ(t)+ \frac{\bJ(t)^2 }{(1-M)}
 \, .  \label{eq:dissipation J}
\end{equation}
Combining (\ref{eq:u0}) -- (\ref{eq:entropy J}) -- (\ref{eq:dissipation J}), we obtain
\[ \bD(t) = \int_{-1}^1 c(t,x) \Big(   \partial_x \log c(t,x) + x +c(t,-1)-c(t,1)\Big)^2\D  x + \frac{1}{(1-M)}\left(\frac{\D }{\D t} \bJ(t)\right)^2\, , \]
and the proof of Lemma \ref{lem:entropy u} is complete.
\end{proof}

\paragraph{A priori estimates} 
We now derive a priori bounds for solutions to \eqref{eq:1D} in the classical sense.

\begin{proposition}[Main a priori estimate]\label{apriori}
Assume that $\int _{-1}^1c_0 \log c_0  \D x < +\infty$. Let $c$ be a classical solution to (\ref{eq:1D}). If  $M<1$, then the following global estimates hold true for all $T>0$:
\begin{eqnarray*}
\sup_{t\in (0,T)}\int _{-1}^1c(t,x) \log c (t,x)  \D x &<& +\infty  \, ,\\
\int _{0}^T\int _{-1}^1 c(s,x)(\partial_x \log c (s,x))^2   \D x\D  s  &<& +\infty\, .
\end{eqnarray*}
\end{proposition}
\begin{proof}
The proof follows from Lemma \ref{lem:entropy u}. Indeed, integrating \eqref{eq:dissipation u_0} in time, it yields that
\begin{eqnarray}
 & &\bH(c|G_M) (t)+\frac{\bJ(t)^2}{2(1-M)} + \int_0^t \int_{-1}^1 c(s,x) \bigg(   \partial_x \log c(s,x)  + x  + c(s,-1)-c(s,1)\bigg)^2\D  x  \D s\nonumber \\
& & +\frac{1}{(1-M)}\int_0^t\left(\frac{\D }{\D t} \bJ(s)\right)^2 \D s = \bH(c|G_M)(0)+\frac{\bJ(0)^2}{2(1-M)} \, . \label{eq:estimate_1}
\end{eqnarray}
Since $\bH(c|G_M)\ge M\log M$, it remains to prove that $c(t,-1)-c(t,1)$ belongs to $L^2$ locally in time.  We first derive the following trace-type inequality: 
\begin{eqnarray}
 \left(c(t,-1)-c(t,1)\right)^2 &=& \left( \int_{-1}^1  \partial_x c(t,x) \D x \right)^2 \nonumber \\
\label{ineq:trace:0}
 & \le &\left(\int_{-1}^1 c(t,x) \D x \right) \left(\int_{-1}^1 c(t,x) \left( \partial_x \log c(t,x) \right)^2  \D x \right) \, .
\end{eqnarray}
Furthermore, recalling that
\begin{equation*}
2 (1-M)\left(c(t,-1)-c(t,1)\right) \bJ(t )= \frac{\D }{\D t} \bJ(t)^2+2 \bJ(t)^2 \, ,
\end{equation*}
and 
\begin{equation*}
\frac{\D }{\D t} \bJ(t) ^2+ \bJ(t)^2 +\left(\frac{\D }{\D t} \bJ(t)\right)^2 = (1-M)^2 \left( c(t,-1)-c(t,1)\right)^2  \ge 0\, ,
\end{equation*}
we deduce that 
\begin{eqnarray}
 & &2 \left(c(t,-1)-c(t,1)\right) \bJ(t ) +\frac{1}{(1-M)}\left(\frac{\D }{\D t} \bJ(t)\right)^2 \nonumber \\
 &=&  (1-M) \left( c(t,-1)-c(t,1)\right)^2  \ge 0\, ,\label{ineq:positivity}
\end{eqnarray}
which together with inequality (\ref{ineq:trace:0}) yields that
\begin{eqnarray}
& &\int_{-1}^1 c(t,x) \Big(   \partial_x \log c(t,x)  + x  + c(t,-1)-c(t,1)\Big)^2\D  x +\frac{1}{(1-M)}\left(\frac{\D }{\D t} \bJ(t)\right)^2
\nonumber \\ &= &
\int_{-1}^1 c(t,x) \left(  \partial_x \log c(t,x) \right)^2\D x + (M-2) \left(c(t,-1)-c(t,1)\right)^2  + 2 \left(c(t,-1)-c(t,1)\right) \bJ(t ) \nonumber \\
 & & + \int_{-1}^1  x^2 c(t,x)\D x - 2M +2\left(c(t,-1)+c(t,1) \right) +\frac{1}{(1-M)}\left(\frac{\D }{\D t} \bJ(t)\right)^2\nonumber \\
& \geq &  \left(M + \frac{1}{M}-2\right) \left(c(t,-1)-c(t,1)\right)^2   - 2M \, . \label{ineq:estimate_1}
\end{eqnarray}

The quantity $M + M^{-1} - 2$ is positive since $M<1$. Hence, using \eqref{eq:estimate_1} and \eqref{ineq:estimate_1} we can prove that $c(t,-1)-c(t,1)$ belongs to $L^2$ locally in time. Next, using again \eqref{ineq:estimate_1} togther with \eqref{ineq:positivity}, we see that
\begin{eqnarray}
& &\int_{-1}^1 c(t,x) \Big(   \partial_x \log c(t,x)  + x  + c(t,-1)-c(t,1)\Big)^2\D  x +\frac{1}{(1-M)}\left(\frac{\D }{\D t} \bJ(t)\right)^2
\nonumber \\ &\ge  &
\int_{-1}^1 c(t,x) \left(  \partial_x \log c(t,x) \right)^2\D x + (M-2) \left(c(t,-1)-c(t,1)\right)^2 -2M\, . \label{ineq:estimate_2}
\end{eqnarray}
Hence,  $\int_{-1}^1 c(t,x) \left(  \partial_x \log c(t,x) \right)^2\D x$ belongs to $L^1$ locally in time.
\end{proof}

\paragraph{Long-time behaviour}\label{sec:StSt}
To prove convergence of $c(t,\cdot)$ towards $G_M $ we develop the following strategy.  We use the previous {\em a priori} estimates which enable to pass to the limit after extraction of a converging subsequence. The main argument (apart from passing to the limit) consists in identifying the possible configurations $c_\infty$ for which the dissipation $\bD$ vanishes. This occurs if and only if both positive terms in \eqref{eq:dissipation u} are zero. Thanks to \eqref{eq:first_momentum}, this means that $\bJ_\infty = (1 - M) \left(c_\infty(-1)-c_\infty(1)\right)$ on the one hand, and on the other hand,
\[ \partial_x \log c_\infty(x)  + x + c_\infty(-1)-c_\infty(1) = 0  \, ,\]
which yields that $c_\infty \equiv G_M$. 

\paragraph{Rate of convergence} 
As it is classical when the equilibrium state is a gaussian function, the natural tool is a logarithmic Sobolev inequality established by Gross in \cite{gross} that we first recall. Although we are dealing here with a non linear problem this method will be fruitful.
\begin{lemma} [Logarithmic Sobolev inequality] \label{LSI(1)} 
Let $\nu(x)dx=\exp(-V(x))dx$ be a measure with smooth density on $[-1,1]$. Assume that $V''(x)\geq 1$ then, for $u\geq 0$ satisfying $\int_{-1}^1 u(x) \D x=\int_{-1}^1 \nu (x) \D x$, we have
\begin{equation*}
\int_{-1}^1 u(x)\log \left(\frac{u(x)}{\nu(x)}\right)\D x\leq \frac{1}{2}\int_{-1}^1 u(x)\left(\partial_x\left(\log\frac{u(x)}{\nu(x)}\right)\right)^2 \D x\, .
\end{equation*}
\end{lemma}

First, recalling \eqref{eq:dissipation u_0} and \eqref{eq:dissipation u}, we deduce that 
\begin{equation*}
\frac{\D }{\D t} \bL(t)=  -\bD(t) = -\bI(c|\Gamma_c) - \frac{1}{(1-M)}\bigg(\Big(c(t,-1)-c(t,1)\Big)(1-M)-\bJ(t)\bigg)^2 \, .
\end{equation*}
Our aim is to apply a logarithmic Sobolev inequality to $\bH(c|\Gamma_c)$. We first observe that 
\begin{eqnarray*}
\frac{\D }{\D t} \bL(t)+2\bL(t)
&=& -\bI(c|\Gamma_{c})- (1-M)\left(c(t,-1)-c(t,1)\right)^2+2\bH(c|G_{M}) \\ 
& &+ 2\left(c(t,-1)-c(t,1)\right) \bJ(t)\,,  
\end{eqnarray*}
and we decompose the relative entropy as follows
\begin{eqnarray*}
\bH(c|G_M)&=&\int_{-1}^1 c(t,x) \log \left(\frac{c(t,x)}{G_M(x)} \right)\D x\\
&=&\int_{-1}^1 c(t,x) \log \left(\frac{c(t,x) }{\Gamma_c(x)}\right)\D x +\int_{-1}^1 c(t,x) \log \left(\frac{\Gamma_c(x)}{G_M(x)}\right) \D x\, .
\end{eqnarray*}
Recalling the definition \eqref{Lyapounov_ss_crit_1}  of $\Gamma_c$ we deduce that
\begin{eqnarray*}
& &\int_{-1}^1 c(t,x) \log \left(\frac{\Gamma_c(x)}{G_M (x)}\right) \D x= \int_{-1}^1  c(t,x) \log \left(\frac{\int _{-1}^1 \exp\left(-y^2/2\right) \D y}{M}\right) \D x\\
& &+\int_{-1}^1  c(t,x)  \log \left(A_c\exp\Big(-\left(c(t,-1)-c(t,1) \right) \, x\Big)\right) \D x\\
& = & M
\log\left(\frac{ \int _{-1}^1 \exp\left(-y^2/2\right) \D y}{\int_{-1}^1\exp \left( - \left( c(t,-1)-c(t,1)\right)  y -\frac{y^2}{2}\right)\D y}\right)-\left(c(t,-1)-c(t,1) \right) \bJ(t)\, ,
\end{eqnarray*}
thanks to the definition \eqref{cte_Lyapounov_ss_crit_1} of $A_c$. Therefore, 
\begin{eqnarray}
2\bH(c|G_M)&=&2\bH(c|\Gamma_{c})+2M\log\left(\frac{ \int _{-1}^1 \exp\left(-x^2/2\right) \D x}{\int_{-1}^1\exp \left( - \left( c(t,-1)-c(t,1)\right)  x -\frac{x^2}{2}\right)\D x}\right)\nonumber \\
& & -2\left(c(t,-1)-c(t,1) \right) \bJ(t)\, .
\label{eq:2entropies}
\end{eqnarray}
Finally, applying a logarithmic Sobolev inequality (Lemma \ref{LSI(1)}) to the measure $\Gamma_{c} (x)\D x$ and using that  $- (1-M)\left(c(t,-1)-c(t,1)\right)^2\le 0$, we obtain 
\begin{eqnarray*}
\frac{\D }{\D t} \bL(t)+2\bL(t)
&=& -\bI(c|\Gamma_{c})- (1-M)\left(c(t,-1)-c(t,1)\right)^2+2\bH(c|\Gamma_{c})  \\
& &+ 2M\log\left(\frac{ \int _{-1}^1 \exp\left(-x^2/2\right) \D x}{\int_{-1}^1\exp \left( - \left( c(t,-1)-c(t,1)\right)  x -\frac{x^2}{2}\right)\D x}\right)\\
&\le & 2M\log\left(\frac{ \int _{-1}^1 \exp\left(-x^2/2\right) \D x}{\int_{-1}^1\exp \left( - \left( c(t,-1)-c(t,1)\right)  x -\frac{x^2}{2}\right)\D x}\right)\\
&=& -2M\log \left(\int_{-1}^1  \exp\left(-\left( c(t,-1)-c(t,1)\right)  x \right)\frac{G_M (x)}{M} \D x\right)\\
& \le & 0\, ,
\end{eqnarray*}
thanks to Jensen's inequality. Hence
\begin{equation*}
\dfrac{\D}{\D t} \bL(t)+2\bL(t)\leq 0\, ,
\end{equation*}
leading to 
\begin{equation*}
\bH(c|G_M)  \leq  \bL(0)\exp(-2t)\, .
\end{equation*}
A rate of convergence for the $L^1$ norm is obtained by using the  Csisz\'ar-Kullback  inequality, \cite{Csiszar}, that we recall now.
\begin{proposition} [Csisz\'ar-Kullback  inequality]
For any non-negative functions  $f,g\in L^1(-1,1) $ such that $\int _{-1}^1  f(x) \D x=\int  _{1}^1  g(x) \D x=M$, we have that 
\begin{equation}\label{CK}
\|f-g\|_1^2\leq 2 M \int_{-1}^1 f(x)\log \left(\frac{f(x)}{g(x)}\right)\D x\, .
\end{equation}
\end{proposition}
Hence, the following decay estimate holds:
\begin{eqnarray*}
\|c(t,x)- G_M(x)\|_{L^1} & \leq & \sqrt{2 M \bL(0)}\exp(-t)\, .
\end{eqnarray*}

\subsection{Critical case $M=1$}\label{sec:M=1}
Similarly, we can prove the existence of a stationary solution.
\begin{lemma}\label{lemma:EScritique}
For $M=1$, equation \eqref{eq:1D} admits a unique stationary solution given by $G_1$.
\end{lemma}
\begin{proof}
The proof is similar to the one of lemma \ref{lemma:ESsscritique}.
\end{proof}

In this part we cannot follow the strategy developped in Section \ref{sec:M<1} since we crucially used $M<1$. In the case $M=1$, from \eqref{eq:u0} it follows that 
\begin{equation}\label{eq:dissipation_critique}
\frac{\D }{\D t} \bH(c|G_1)=  -\bI(c|\Gamma_c) + \left(c(t,-1)-c(t,1)\right) \bJ(t)  \, ,
\end{equation}
and
\begin{eqnarray*}
\bH(c|G_1)&=&\bH(c|\Gamma_{c})+\log\left(\frac{ \int _{-1}^1 \exp\left(-x^2/2\right) \D x}{\int_{-1}^1\exp \left( - \left( c(t,-1)-c(t,1)\right)  x -\frac{x^2}{2}\right)\D x}\right)\\
& & -\left(c(t,-1)-c(t,1) \right) \bJ(t)\, .
\end{eqnarray*}
Then, together with the logarithmic Sobolev inequality $2\bH(c|\Gamma_c) \leq \bI(c|\Gamma_c)$, we deduce that
\begin{equation*}
\frac{\D }{\D t} \bH(c|G_1)\le   -\bH(c|\Gamma_c) -\bH(c|G_1)+\log\left(\frac{ \int _{-1}^1 \exp\left(-x^2/2\right) \D x}{\int_{-1}^1\exp \left( - \left( c(t,-1)-c(t,1)\right)  x -\frac{x^2}{2}\right)\D x}\right) \, .
\end{equation*}
Consequently, using again Jensen's inequality, it follows that 
\begin{equation}\label{ineq_critique_entropie}
\frac{\D }{\D t} \bH(c|G_1)\le  -\bH(c|G_1)\le 0 \, .
\end{equation}
Consequently we deduce that $0\le \bH(c|G_1)(t)\le \bH(c|G_1)(0)$, hence 
\[\int_{-1}^1 c(t,x) \log c(t,x) \D x \leq C_0 \, , \quad {\rm a.e.}\; t\in (0,+\infty)\, .\]

\paragraph{A priori bound}
To obtain a control on the dissipation of entropy, we follow the strategy developed in \cite{CalvezMeunier_Siam}. Consider the even function $\Lambda : \R \rightarrow\R_+$ such that $\Lambda(0) = 0$, and $\Lambda'(u) = (\log(u))_+^{1/2}$ for $u>0$. Then, it is non-increasing on $(-\infty,0)$, non-decreasing on $(0,+\infty)$, convex and superlinear, and for all $D$, there exists $A\in \R_+$ such that for any $u \in (-\infty,-A)\cup (A,+\infty)$, $\Lambda(u)^2 \geq (1+D) C_0 u^2$. Using again a trace-type inequality, we get 

\begin{eqnarray}
 \Lambda\left(c(t,-1)-c(t,1)\right)^2 &=& \left( -\int_{-1}^1  \partial_x  \Lambda(c(t,x)-c(t,1)) \D x \right)^2 \, , \nonumber\\
 & = & \left( -\int_{-1}^1  \Lambda'(c(t,x)-c(t,1)) c(t,x) \partial_x \log(c(t,x)) \D x \right)^2 \, , \nonumber\\
 &\le & \left( \int_{-1}^1 c(t,x) \left|\Lambda'(c(t,x)-c(t,1))\right|^2 \D x \right)\left( \int_{-1}^1 c(t,x) (\partial_x \log(c(t,x)))^2  \D x \right)\, \nonumber\\
 & \le & \left( \int_{-1}^1 c(t,x) \log(c(t,x))_+
\D x \right)\left( \int_{-1}^1 c(t,x) (\partial_x \log(c(t,x)))^2  \D x \right)\, \nonumber\\
&\le & C_0 \left( \int_{-1}^1 c(t,x) (\partial_x \log(c(t,x)))^2  \D x \right)\, . \label{ineq:lambda}
\end{eqnarray}

Moreover, recall that we have

\begin{eqnarray}
\frac{\D \bH(c|G_1)(t)}{\D t } &=&\alpha(t) \bJ(t) - \int_{-1}^1 c(t,x) \left(\partial_x \log(c(t,x)) + \alpha(t) + x \right)^2 \D x\, , \label{eq:entropie_critique}\\
&= &\alpha(t)^2  - \alpha(t) \bJ(t)  +2 (1 - c(t,1) - c(t,-1)) - \bK(t) - \int_{-1}^1 c(t,x) \left(\partial_x \log(c(t,x))\right)^2 \D x\, , \nonumber \\
&\leq &  \alpha(t)^2  - \alpha(t) \bJ(t)  +2  - \int_{-1}^1 c(t,x) \left(\partial_x \log(c(t,x))\right)^2 \D x\, , \nonumber
\end{eqnarray}
with $\bK(t) = \int_{-1}^1 x^2 c(t,x) \D x \geq 0$.
Combining \eqref{ineq_critique_entropie} and \eqref{ineq:lambda}, it leads to 

\begin{displaymath}
\frac{\D \bH(c|G_1)(t)}{\D t }\leq \left\lbrace 
\begin{array}{ll}
0 & \text{ if } |\alpha(t)| \leq A\,\\
-\frac{ \Lambda(\alpha(t))^2}{C_0} + \alpha(t)^2 +2 - \alpha(t)\bJ(t)\leq 
-D\alpha(t)^2 + 2-\alpha(t)\bJ(t)  &\text{ if } |\alpha(t)| > A\,.
\end{array}\right.
\end{displaymath}

We can assume $A>\sqrt{2}$, so that $\alpha(t)^2>2$. In the case $|\alpha(t)| > A$, as $\bJ(t) = \bJ(0) e^{-t}$, we have $ - \alpha(t) \bJ(t) \leq \alpha(t)^2 |\bJ(0)|$, and 
\begin{displaymath}
-D\alpha(t)^2 + 2-\alpha(t)\bJ(t) \leq -(D-1-|\bJ(0)|)\alpha(t)^2 \, .
\end{displaymath}

Take $D = 2+ |\bJ(0)|$. Then, 
\begin{displaymath}
\frac{\D \bH(c|G_1)(t)}{\D t }\leq \left\lbrace 
\begin{array}{ll}
0 & \text{ if } |\alpha(t)| \leq A\,\\
-\alpha(t)^2  &\text{ if } |\alpha(t)| > A\,.
\end{array}\right.
\end{displaymath}

Now, on the set  $E = \{ t:|\alpha(t)| >A \}$, we have 
\begin{equation}\label{ineq:alpha_critique}
\int_E \alpha(t)^2 \D t \leq \bH(c|G_1)(0)\,,
\end{equation}
giving a $L^2$ control on $\alpha(t)$. 

Finally, using \eqref{eq:entropie_critique} and \eqref{ineq:alpha_critique}, we prove that $\int_{0}^t\int_{-1}^1 c(s,x) \left(\partial_x \log(c(s,x))\right)^2 \D x \D t$ is bounded for all $t\in (0,T)$. Then, proposition \ref{apriori} is still verified, and we can similarly prove the global existence of weak solutions in the critical case. 

\paragraph{Long-time behaviour}
The convergence of $c(t,\cdot)$ towards $G_M $ is proved as in the subcritical case. 

\paragraph{Rate of convergence}
By \eqref{ineq_critique_entropie}, we know that 
\begin{equation*}
\bH(c|G_1)(t) \le \bH(c|G_1)(0) e^{-t}\, .
\end{equation*}
Hence, the Csisz\'ar-Kullback  inequality leads to the following estimation:
\begin{eqnarray*}
\|c(t,x)- G_1(x)\|_{L^1} & \leq & \sqrt{2 M \bH(c|G_1)(0)}\exp(-t/2)\, .
\end{eqnarray*}

\subsection{The Super-critical case $M>1$}\label{sec:M>1}

We prove now that depending on the mass, the problem \eqref{eq:1D} admits one or three stationary solutions. Then, we prove that if the initial condition is asymmetric enough, a blow-up occurs in finite time, and we are able to give a quantitative criterion for its appearance. The proof is based on the blow-up analysis for a modified problem, followed by the use of a concentration-comparison principle to deduce the blow-up in our case. Numerical simulations displayed in Figure \ref{simus:cas_direct} show the appearance of a blow-up.


\subsubsection{Stationary solutions}
We establish now the following lemma.

\begin{lemma}\label{lemma:ESsurcritique}
Denote $M_0 =\frac{1}{2} \int_{-1}^1 exp\left(-\frac{x^2 -1}{2}\right) \D x >1$. Then, 
\begin{itemize}
\item[$\bullet$] for $M\in (1,M_0)$, equation \eqref{eq:1D} admits exactly three stationary states: the symmetric solution $G_M$, and two asymmetric solutions $G_{\pm \alpha}$, with
\begin{equation*} 
G_{\alpha}(x) =  \frac{\alpha }{1-\exp(-2\alpha)} \exp\Big(-\alpha (x+1) - \frac{x^2-1}{2}\Big)\, ,
\end{equation*}  
 for some $\alpha >0$ defined by the mass constraint $\int_{-1}^1 G_{\alpha}(x)\D x = M$.
 \item[$\bullet$] for $M\geq M_0$, $G_M$ is the only stationary solution.
\end{itemize}
\end{lemma}
\begin{proof}
We have seen in the proof of lemma \ref{lemma:ESsscritique} that the only possible stationary solutions are of the form $G_M$ and $G_\alpha$. It is straightforward that for all $M>1$, $G_M$ is solution and satisfies the mass constraint.
 For $G_\alpha$, let us denote $M_\alpha = \int_{-1}^1 G_\alpha(x) \D x$. We know that $M_\alpha >1$. It remains to characterize the set $I = \{ M_\alpha,\, \alpha >0\}$ of attainable mass values. As we were not able to characterize $I$ explicitly, we performed a numerical simulation of $M_\alpha$ as a function of $\alpha$ (see figure \ref{simu:Alpha_M}, a)). Notice that $M_0$ is defined such that $\lim_{\alpha \rightarrow 0} G_{\alpha} = G_{M_0}$, leading to the result. 
\end{proof}

\subsubsection{Blow-up for a modified equation}

In this section we consider the following modified equation of \eqref{eq:1D}:
\begin{equation}\label{eq:1D_modified}
\partial _t c(t,x) =  \partial_x\Big( \partial _{x} c(t,x)+ \left(-1 + c(t,-1)-c(t,1)\right) c(t,x)\Big) \, ,\\
\end{equation}
where $c(t,x)$ is defined for $t>0$ and $x\in (-1,1)$, together with zero-flux boundary conditions:
\begin{equation}\label{eq:1D_modified_bord}
\begin{cases}
\partial _x c(t,-1)+\left(c(t,-1)-c(t,1) -1\right) c(t,-1) =0 \, , \\
\partial _x c(t,1)+ \left(c(t,-1)-c(t,1)-1\right) c(t,1)   =0  \, ,
\end{cases}
\end{equation}
and an initial condition: $c(t=0,x)=c_0(x)$. The total mass will also be denoted by $M$, i.e. $M=\int_{-1}^1 c(t,x) \D x$.
 
In this part we will prove that solutions to \eqref{eq:1D_modified} - \eqref{eq:1D_modified_bord} blow-up in finite time when mass is super-critical $M > 1$ and $c_0$ is decreasing and satisfies $c_0(-1)-c_0(1)>1$. To do so we prove that the first momentum shifted in $x=-1$ of $c$, $\tilde{\bJ}(t) = \int_{-1}^1   (x+1) c(t,x)\D  x$ cannot remain positive for all time. This technique was first used by Nagai \cite{Nagai}, then by many authors in various contexts.
In a first step, in Lemma \ref{lem:decroissant}, we establish that the assumption that $c_0$ is a decreasing function such that $c_0(-1)-c_0(1)>1$ and that the $\tilde{\bJ}(0)$ is sufficiently small guarantee that $c(t, ·)$ is also decreasing and satisfies $c(t,-1)-c(t,1)>1$ for any existence time $t > 0$. In a second step, in Proposition \ref{faibleBDP1:finite}, we prove the blow-up character.
\begin{lemma}\label{lem:decroissant}
Assume that $M>1$. Assume in addition that the initial condition $c_0$ of \eqref{eq:1D_modified} is decreasing, satisfies $c_0(-1)-c_0(1)>1$ and $\int_{-1}^1 (x+1) c_0(x) \D x <(M-1)/2$. Then any solution $c$ to \eqref{eq:1D_modified}, if it exists, is non-increasing and satisfies $c(t,-1)-c(t,1)>1$.   
\end{lemma}
\begin{proof}
Since we supposed that $c_0(-1)-c_0(1)>1$, it remains true at least until a time $t_0 \in (0,T)$, where $T$ is the existence time. We choose the maximal $t_0$ possible. For all $t\in [0,t_0]$, $c(t,\cdot)$ is decreasing. In fact the derivative $u(t,x)= \partial_x c(t,x)$, satisfies the following parabolic type equation without any source term:
\begin{equation}\label{eq:1D_modified_2}
\partial _t u(t,x) =  \partial _{xx} u(t,x)+\partial_x\Big( \left(-1 + c(t,-1)-c(t,1)\right) u(t,x)\Big) \, .
\end{equation}
The solution is initially non-positive, and also initially non-positive on the boundary due to \eqref{eq:1D_modified_bord} and the assumption $c_0(-1)-c_0(1)>1$.  From classical strong maximum principle, we deduce that $c(t,\cdot)$ is decreasing for all $t\in [0,t_0]$.

Therefore, for all $t\in [0,t_0]$, $-\partial_x c(t,x)/\left(c(t,-1)-c(t,1)\right)$ is a probability density. From Jensen's inequality, we deduce the following interpolation estimate:
\[\left(\int_{-1}^1 (x+1) \frac{-\partial_x c(t,x)}{c(t,-1)-c(t,1)}\D  x\right)^2  \leq 
\int_{-1}^1  (x+1)^2 \frac{-\partial_x c(t,x)}{c(t,-1)-c(t,1)}\D  x\, ,\]
hence, using that $c(t,-1)>c(t,1)$ for any time $t\in [0,t_0]$, it follows that
\begin{equation}\label{interpol}
\left(M- 2c(t,1) \right)^2  \leq \left( c(t,-1)-c(t,1) \right) \left( 2\int_{-1}^1   (x+1)  c(t,x) \D  x - 4c(t,1)\right)\, .
\end{equation}

Consequently, for all $t$ in $[0,t_0]$, the first momentum shifted in $x=-1$, $\tilde{\bJ}(t) = \int_{-1}^1   (x+1) c(t,x)\D  x$, which is non-negative and such that  $2\tilde{\bJ}(t) \ge 4 c(t,1)$, for all $t$ in $[0,t_0]$, thanks to  (\ref{interpol}),  satisfies:
\begin{eqnarray}
\frac{\D }{\D t} \tilde{\bJ}(t)  & \leq &  M + (1-M) \left(c(t,-1)-c(t,1)\right) \leq (1-M) \frac{(M - 2c(t,1))^2}{2\tilde{\bJ}(t) - 4 c(t,1)} +M \nonumber \\
& \leq &(1-M) \frac{M^2 - 4M c(t,1)}{2\tilde{\bJ}(t)} +M\, , \label{eq:premier_moment} 
\end{eqnarray}
as $M>1$. Using again that $2\tilde{\bJ}(t) \ge 4 c(t,1)$, we deduce that
\begin{eqnarray*}
\frac{\D }{\D t} \tilde{\bJ}(t)  &\leq & \frac{M(1-M)}{2\tilde{\bJ}(t)} \left(M - 2\tilde{\bJ}(t) \right) +M\leq \frac{M^2(1-M)}{2\tilde{\bJ}(t)} \left( 1-\frac{2\tilde{\bJ}(t)}{M-1}\right)\, .
\end{eqnarray*}
Since $\tilde{\bJ}(0)<(M-1)/2$ and $\tilde{\bJ}(t)\ge 0$ we deduce that for all $t \in [0,t_0]$: $\frac{\D }{\D t} \tilde{\bJ}(t) \le 0$, hence
$\tilde{\bJ}(t) \le \tilde{\bJ}(0)<\frac{M-1}{2}$.
Consequently, using  (\ref{interpol}), it follows that for all $t$ in $[0,t_0]$, 
\begin{equation*}
c(t-1)-c(t,1) \ge \frac{M\left( M-2\tilde{\bJ}(0) \right)}{2\tilde{\bJ}(0)} \ge  \frac{M}{M-1}>1 \, ,
\end{equation*} 
hence $t_0=T$ is the existence time of $c$.
\end{proof}

We can now prove a blow-up result for \eqref{eq:1D_modified} - \eqref{eq:1D_modified_bord}.
\begin{proposition}\label{faibleBDP1:finite}
Assume $M>1$ and that the first moment shifted in 1 is initially small: $\tilde{\bJ}(0) <  (M-1)/2$. Assume in addition that $c_0$ is decreasing and satisfies $c_0(-1)-c_0(1)>1$. Then the solution to \eqref{eq:1D_modified} - \eqref{eq:1D_modified_bord} with initial data $c(0,x) = c_0(x)$ blows-up in finite time.
 \end{proposition} 
\begin{proof}
From Lemma \eqref{lem:decroissant} it follows that 
\begin{equation}
\frac{\D }{\D t} \tilde{\bJ}(t) \le \frac{M^2(1-M)}{2\tilde{\bJ}(t)} \left( 1-\frac{2\tilde{\bJ}(t)}{M-1}\right)\le  \frac{M^2(1-M)}{2\tilde{\bJ}(t)} \left( 1-\frac{2\tilde{\bJ}(0)}{M-1}\right)\, .
\end{equation}
Hence, since $\tilde{\bJ}(0) < (M-1)/2$,  we deduce the inequality
\begin{eqnarray}
\tilde{\bJ}(t)  & \leq &\tilde{\bJ}(0) + \frac{(1 - M)M^2}{2} \left( 1-\frac{2\tilde{\bJ}(0)}{M-1}\right)\int_0^t \frac{1}{\tilde{\bJ}(s)}\D s \label{eq:moment0}\, .
\end{eqnarray}
Now, $\tilde{\bJ}$ being non-increasing on $[0,t_0]$, we deduce that 
\begin{equation}
\tilde{\bJ}(t)   \leq \tilde{\bJ}(0) + \frac{(1 - M)M^2}{2} \left( 1-\frac{2\tilde{\bJ}(0)}{M-1}\right) \frac{1}{\tilde{\bJ}(0)}t_0\, ,
\end{equation}
and with $M>1$, we deduce that the maximal time of existence $t_0=T^*$ of a solution is necessarily finite.
Finally, following \cite{JL}, it can be proved that 
\[ \lim_{K\to +\infty} \left( \sup_{t\in(0,T^*)} \int_{-1}^1 (c(t,x)-K)_+\D x\right) >0\, , \]
showing that the solution becomes singular in $T^*$. Otherwise a truncation method enables to prove local existence by replacing $c$ with $(c - K)_+$ for $K$ sufficiently large.
\end{proof}

\subsubsection{Blow-up in the Super-critical case}

In this part, we prove Theorem \ref{th:BU}. To do so, following \cite{Lepoutre_Meunier_Muller_JMPA}, we state a so-called concentration-comparison principle on the equation obtained from \eqref{eq:1D} after space integration. This principle together with the use of a subsolution that blows-up allow proving the blow-up character of solutions to \eqref{eq:1D} above the critical mass.

\begin{lemma}\label{lem:ccp_sub_super}
Let $C,\bar C, \underline{C}$ be non-decreasing (in space)  functions in $C^1(0,T;C^2([-1,1]))$ satisfying 
\begin{equation}\label{main_integre}
\begin{cases}
\partial_t C(t,x)-\partial_{xx}C(t,x)-\Big(x+\partial_xC(t,-1)-\partial_xC(t,1)\Big)\partial_xC(t,x)=0\, ,\\
\partial_t \bar C(t,x)-\partial_{xx}\bar C(t,x)-\Big(x+\partial_x\bar C(t,-1)-\partial_x \bar C(t,1)\Big)\partial_x\bar C(t,x)\geq 0\, ,\\
\partial_t \underline{C}(t,x)-\partial_{xx}\underline{C}(t,x)-\Big(x+\partial_x\underline{C}(t,-1)-\partial_x\underline{C}(t,1)\Big)\partial_x\underline{C}(t,x)\leq 0\, ,\\
C(t,-1)=\bar C(t,-1)= \underline{C}(t,-1)=0\, .
\end{cases}
\end{equation}
Assume that $ \underline{C}(0,\cdot) \le C(0,\cdot) \le \bar C(0,\cdot)$ and that $\partial_x\bar C(0,-1)- \partial_x\bar C(0,1)>\partial_x\underline{C}(0,-1)-\partial_x\underline{C}(0,1)$, then the following inequality holds true for all $0<t\le T$:
\[
\underline{C}(t,.)< \bar C(t,.).
\]
\end{lemma}
\begin{proof}
From \eqref{main_integre}, we deduce that $\delta C=\bar C- \underline{C}$ satisfies the parabolic inequation
\begin{eqnarray*}
& &\partial_t \delta C(t,x)-\partial_{xx}\delta C(t,x)-\Big(x+\partial_x\bar C(t,-1)-\partial_x\bar C(t,1)\Big)\partial_x\delta C(t,x) \\
& &\ge\bigg( \Big(\partial_x\bar C(t,-1)-\partial_x\bar C(t,1)\Big)- \Big(\partial_x\underline{C}(t,-1)-\partial_x\underline{C}(t,1)\Big) \bigg)\partial_x\underline{C}(t,x)\, ,
\end{eqnarray*}
with $\bar C(t,-1)= \underline{C}(t,-1)=0$.
Since we supposed that $\partial_x\bar C(0,-1)- \partial_x\bar C(0,1)>\partial_x\underline{C}(0,-1)-\partial_x\underline{C}(0,1)$, it remains true at least until a time $T' \in (0,T)$, and we choose the maximal $T'$ possible. Since $\partial_x\underline{C}\geq 0$, on the time interval $[0,T']$, we have
\[
\begin{cases}
\partial_t \delta C(t,x)-\partial_{xx}\delta C(t,x)-\Big(\partial_x\bar C(t,-1)-\partial_x\bar C(t,1)\Big)\partial_x\delta C(t,x) - x\partial_{x}\delta C(t,x)\geq 0\, ,\\
\bar C(t,-1)= \underline{C}(t,-1)=0\, .
\end{cases}
\]
Hence, by strong maximum principle \cite{Evans}, $\bar C> \underline{C}$ on $(0,T')\times (-1,1)$. Furthermore, by Hopf Lemma (see \cite{Evans}), we also have 
\[
\begin{cases}
\partial_x \delta C(T',-1)=\partial_x\bar C(T',-1)-\partial_x\underline{C}(T',-1)>0\, ,\\
\partial_x \delta C(T',1)=\partial_x\bar C(T',1)-\partial_x\underline{C}(T',1)<0\, .
\end{cases}
\]
As $T'$ is maximal we immediately conclude that $T'=T$.
\end{proof}

\begin{proof}{Proof of Theorem \ref{th:BU}.}
Let $\underline{c}_0$ be a decreasing function such that $\underline{c}_0(x) \le c_0(x)$ for all $x \in [-1,1]$ and which satisfies $\int_{-1}^1 \underline{c}_0(x) \D x >1$. Assume in addition that $c_0(-1)-c_0(1)>\underline{c}_0(-1)-\underline{c}_0(1)>1$. Finally assume that $\int_{-1}^1 (x+1) \underline{c}_0(x) \D x <(M-1)/2$, where we recall that $M=\int_{-1}^1 c(t,x) \D x>1$. First according to Proposition \ref{faibleBDP1:finite}, we know that the solution $\underline{c}$ to \eqref{eq:1D_modified} - \eqref{eq:1D_modified_bord} with initial data $\underline{c}(0,x) = \underline{c}_0(x)$ blows-up in finite time. 

We will now prove that the distribution function
$\underline{C}(t,x)=\int_{-1}^x \underline{c}(t,y)\D y$ satisfies 
\[\partial_t \underline{C}(t,x)-\partial_{xx}\underline{C}(t,x)-\Big(x+\partial_x\underline{C}(t,-1)-\partial_x\underline{C}(t,1)\Big)\partial_x\underline{C}(t,x)\leq 0\, .\]
Indeed integrating \eqref{eq:1D_modified} - \eqref{eq:1D_modified_bord} in space, one obtains
\[ \partial _t \underline{C}(t,x) -\partial_{xx} \underline{C}(t,x) - \Big( -1+\partial_x\underline{C}(t,-1)- \partial_x \underline{C}(t,1)\Big) \partial_x \underline{C}(t,x)  =0\, .  \]
Hence,
\[ \partial _t \underline{C}(t,x) -\partial_{xx} \underline{C}(t,x) - \Big( x+\partial_x\underline{C}(t,-1)- \partial_x \underline{C}(t,1)\Big) \partial_x \underline{C}(t,x)  =-(x+1) \partial_x  \underline{C}(t,x) \le 0\, ,  \]
where the last inequality follows from the non-decreasing character of $\underline{C}(t,\cdot)$ together with $x\ge -1$. 
Hence, from Lemma \ref{lem:ccp_sub_super}, it follows that $\underline{C}(t,.)< \bar C(t,.)$ for all time below the existence time $T$, where $\bar C(t,.)$ is any function defined in \eqref{main_integre}. In particular, for $c$ solution to \eqref{eq:1D} with initial data $c(0,x) = c_0(x)$, we can set $\bar C (t,x)=\int_{-1}^x c(t,y) \D y$, which yields the blow-up character of $c$. 
\end{proof}

\section{The model with dynamical exchange of markers at the boundary: prevention of blow-up and asymptotic behaviour} \label{sec:ODE/PDE}

In Section \ref{sec:M>1}, we proved that finite time blow-up occurs in the basic model \eqref{eq:1D} when mass is super-critical $M>1$. Such a behaviour is not realistic from a biological viewpoint. Here we modifiy the model model \eqref{eq:1D} by considering markers that are sticked to the boundary and thus create the attracting drift. More precisely, in this part we will study the stationary states associated with the following model:
\begin{equation} \label{eq:Attach}
\begin{cases}
\partial _t c(t,x)=  \partial _{xx} c(t,x) + \partial _x \left(\left( x +  \mu_- (t)-\mu_+(t)\right) c(t,x)\right) \, , \quad t >0\, , \, x\in (-1,1) \\ 
\frac{\D }{\D t}\mu_-(t)=   c(t,-1)-  \mu_-(t)\, ,\\
\frac{\D }{\D t}\mu_+(t)=   c(t,1)-  \mu_+(t)\, ,
\end{cases}
\end{equation}
together with the flux condition at the  boundary:
\begin{equation} \label{eq:BC2dim}
\begin{cases}
\partial _x c (t,-1)+ \left(-1+\mu_- (t)-\mu_+(t)\right) c (t,-1)=\frac{\D }{\D t}\mu_-(t)\, , \\
\partial _x c (t,1)+  \left(1+\mu_- (t)-\mu_+(t)\right)  c (t,1)=-\frac{\D }{\D t}\mu_+(t)\, .
\end{cases}
\end{equation}
The quantities $\mu_\pm$ represent the concentrations of markers which are sticked to the boundary and thus create the attracting drift $\mu_- (t)-\mu_+(t)$. 
The dynamics of $\mu_\pm$ is driven by simple attachment/detachment kinetics. The mass of molecular markers is shared between the free particles $c(t,x)$ and the particles on the boundary $\mu_\pm(t)$. The boundary condition (\ref{eq:BC2dim}) guarantees conservation of the total mass:
\begin{equation} 
\int_{-1}^1 c(t,x)\D x + \mu_-(t)+\mu_+(t) = M\, . \label{eq:mass conservation dim}\end{equation} 
From (\ref{eq:mass conservation dim}), we easily deduce that finite time blow-up cannot occur since $|\mu_-(t)-\mu_+(t)|$ is bounded by $M$. As a consequence, the problem under study is a conservative convection-diffusion equation with bounded advection term. We denote by $m(t)$ the mass of free particles: 
\begin{equation*}
m(t) = \int_{-1}^1 c(t,x)\D x\, .
\end{equation*}
The conservation of mass reads 
\[
\frac{\D }{\D t}m(t) + \frac{\D }{\D t}\left(\mu_-(t)+\mu_+(t)\right) = 0\, .\]


Stationary solutions $(c,\mu_-,\mu_+)$ are characterized  by:
\begin{equation*}
\begin{cases}
c(x) = c(-1) e^{-\frac{x^2 - 1}{2} - \alpha(x+1)} \, , \, x\in (-1,1) \\
\alpha = c(-1) (1 - e^{-2\alpha})\, , \\
\int_{-1}^1 c(x) \D x +c(-1) + c(1)= M\,, \\
(\mu_-,\mu_+) = (c(-1) , c(1))\, ,
\end{cases}
\end{equation*}
where we have denoted $\alpha = c(-1) - c(1)$.

We establish the following lemma.

\begin{lemma}\label{lemma:ESattach_detach}
Denote $M_0 =1 + \frac{1}{2} \int_{-1}^1 exp\left(-\frac{x^2 -1}{2}\right) \D x >1$. Then, 
\begin{itemize}
\item[$\bullet$] for $M\leq M_0$, the problem \eqref{eq:Attach} - \eqref{eq:BC2dim} admits a unique symmetric stationary solution $G_M$ defined by:
\begin{equation}
G_M(x) = \frac{M}{2  + \int_{-1}^1 e^{-\frac{x^2 - 1}{2}} \D x}  e^{-\frac{x^2 - 1}{2}} \D x\, .
\end{equation}
 \item[$\bullet$] for $M >M_0$, there are two other asymmetric solutions $G_{\pm \alpha}$, with
\begin{equation*} 
G_{\alpha}(x) =  \frac{\alpha }{1 - e^{-2 \alpha}} e^{-\frac{x^2 - 1}{2} -\alpha (x+1)}\, ,
\end{equation*}  
 for some $\alpha >0$ defined by the mass constraint 
\begin{equation*}
M_\alpha := \frac{\alpha}{1 - e^{-2\alpha}}\left(2 + \int_{-1}^1 e^{-\frac{x^2 - 1}{2} -\alpha (x+1)} \D x\right)  -\alpha  = M\, .
\end{equation*} 
\end{itemize}
\end{lemma}

\begin{proof}
A simple computation yields that any stationary solution is either of the form $G_M$ or of the form $G_\alpha$. It is straightforward that for all $M>0$, $G_M$ is solution and satisfies the mass constraint.
 For $G_\alpha$, we need to characterize the set $I = \{ M_\alpha,\, \alpha >0\}$. It can be proved that for all $\alpha >0$, $M_\alpha=M_{-\alpha}$, and $M_\alpha \sim_{+\infty} \alpha$, and that $\lim_{\alpha \rightarrow 0} M_\alpha = M_0$. As before, we performed a numerical simulation of $M_\alpha$ as a function of $\alpha$ (see figure \ref{simu:Alpha_M}, b)), leading to $I=(M_0,+\infty)$, hence the result.
\end{proof}

\begin{figure}\label{simu:Alpha_M}
\centering
\includegraphics[scale=0.3,trim = 0mm 130mm 0mm 0mm]{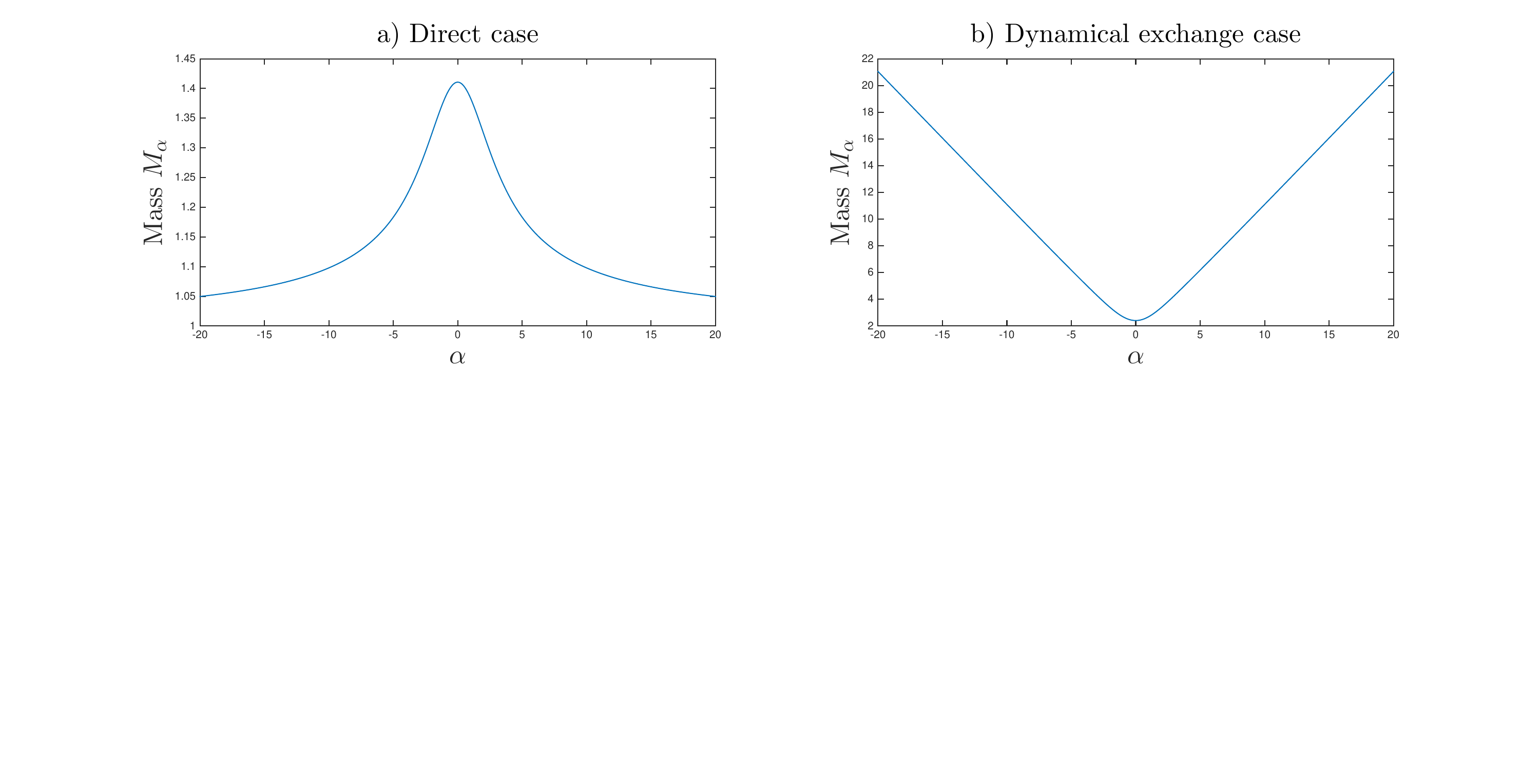}
\caption{Mass of $G_\alpha$ as a function of $\alpha$ for a) the direct case ; b) the dynamical exchange case.}
\end{figure}

Numerical simulations for this model are presented in figure \ref{simus:cas_attach}.

\begin{figure}\label{simus:cas_attach}
\centering
\includegraphics[scale=0.3]{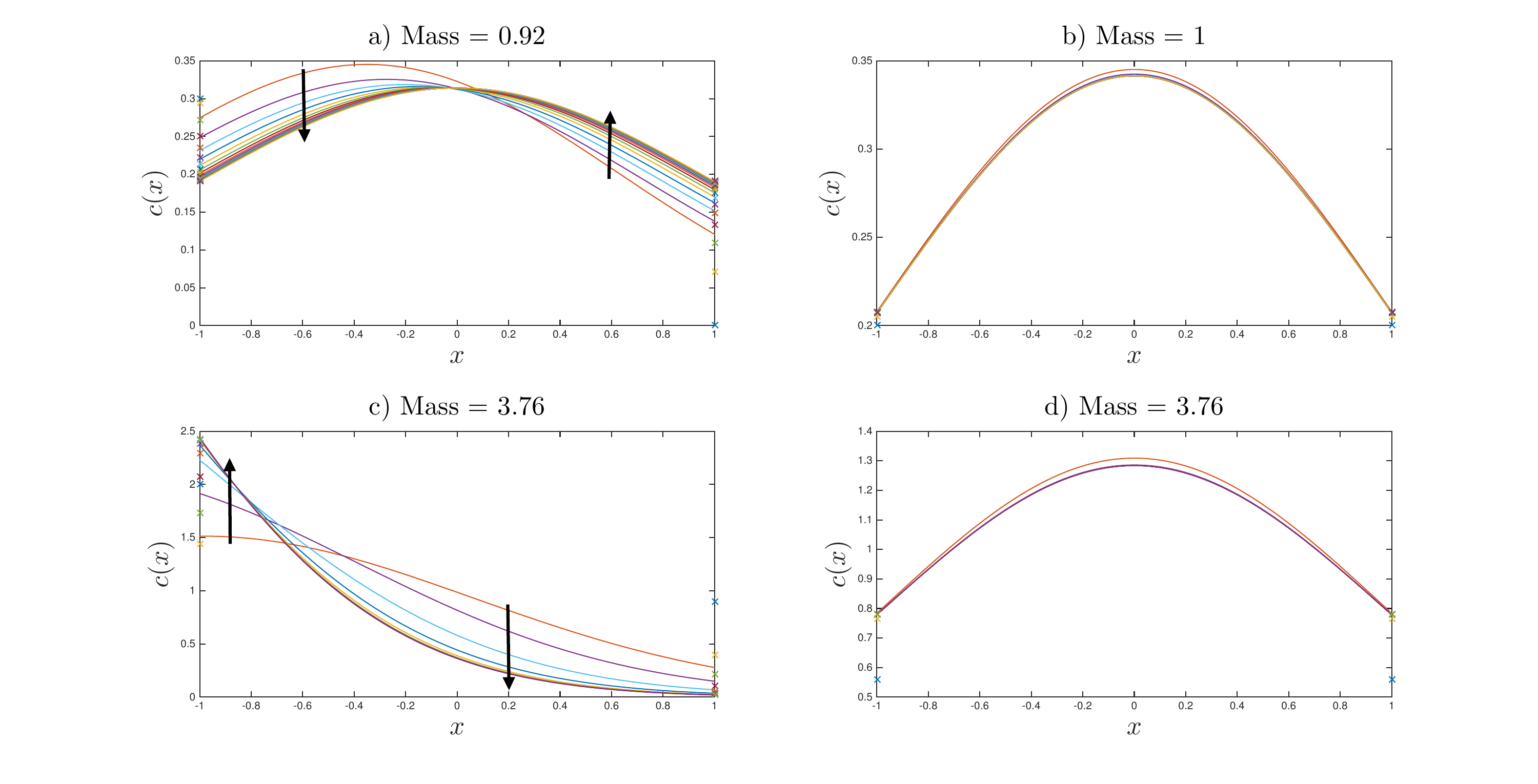}
\caption{Numerical simulations of the spatial concentration profile $c(x)$ of the marker, for the dynamical exchange model. Each plot corresponds to a different initial profile and mass :  a) sub-critical case ; b) critical case ; c), d) super-critical case. Each curve represents the concentration profile at a specific time. Parameters: $T=20$ ; $dt=10^{-2}$ ; $dx = 2* 10^{-3}$.}
\end{figure}

\section{Perspectives}
In this work, we presented a model of 1D cell migration based on the diffusion and advection of a molecular marker inside the cell, that itself exerts a feedback on this dynamics. The marker's spatial repartition is supposed to be characteristic of the polarisation state of the cell. 
The resulting equation, a non linear and non local Fokker-Planck equation, was shown to admit a dichotomy behaviour at equilibrium. Below and at the critical mass, polarisation is not possible. Between the critical mass and a limiting mass $M_0$, the system admits two polarised equilibria. Moreover, we have proved that for initial concentrations "steep enough", a blow up occurs in finite time, corresponding to a polarised state, but mainly showing that the model lacks regularity and that there is room for improvement. 
We next presented a more realistic model, where attachment-detachment kinetics of the marker at the membrane is taken into account. In this case, we were able to prove again a dichotomy behaviour for the existence of asymmetric steady states, and that the additional dynamics at the boundary is enough to prevent the blow up appearance. \par
 
A few questions remain today unresolved, and will be the object of a future work. The stability of either symmetric or asymmetric stationary states ought to be discussed, as prospective numerical simulations suggest non trivial results. Also, in the super critical case, it is of interest to search for a quantitative criterion to discriminate between the three stationary solutions. Finally, the controllability of the cell velocity can be studied. 

This model was designed to investigate the initiation of polarisation during motion. In order to describe a whole displacement, a stochastic instability could be included in the marker's dynamics. Then, the dichotomy's result could be investigated at long time scales. The question of an external signal perturbing the molecular dynamics is of similar nature. 

Moreover, a natural continuation consists in studying the corresponding 2D model, and the equivalent problem on a free boundary domain, in order to describe geometrical feedbacks of the cell on its motion. 


\medskip

\noindent{\em Acknowledgement: the authors are very grateful to V. Calvez, B. Maury, A. Mogilner and M. Piel 
for very helpful discussions. }

\bibliographystyle{plain}
\def\cprime{$'$} \def\lfhook#1{\setbox0=\hbox{#1}{\ooalign{\hidewidth
  \lower1.5ex\hbox{'}\hidewidth\crcr\unhbox0}}} \def\cprime{$'$}
  \def\cprime{$'$} \def\cprime{$'$} \def\cprime{$'$} \def\cprime{$'$}

\end{document}